\documentclass{amsart}

\usepackage{amsmath}
\usepackage{amssymb}
\usepackage{amscd}
\usepackage{hyperref}

\newtheorem{theorem}{Theorem}[section]
\newtheorem{thm}[theorem]{Theorem}
\newtheorem{cor}[theorem]{Corollary}
\newtheorem{lem}[theorem]{Lemma}

\theoremstyle{definition}
\newtheorem{defn}[theorem]{Definition}

\newtheorem{rem}[theorem]{Remark}
\newtheorem{constr}[theorem]{Construction}

\newtheorem{hypo}[theorem]{Hypothesis}

\theoremstyle{remark}

\newcommand{\mbb}{\mathbb}
\newcommand{\QQ}{\mbb{Q}}

\newcommand{\PP}{\mbb{P}}

\newcommand{\FF}{\mbb{F}}

\newcommand{\mc}{\mathcal}

\newcommand{\mcN}{\mc{N}}

\newcommand{\mcX}{\mc{X}}

\newcommand{\mfm}{\mathfrak{m}}

\newcommand{\OO}{\mc{O}}

\newcommand{\wht}{\widehat}
\newcommand{\whts}{\wht{s}}

\newcommand{\SP}{\text{Spec }}

\newsavebox{\sembox}
\newlength{\semwidth}
\newlength{\boxwidth}

\newcommand{\Sem}[1]{%
\sbox{\sembox}{\ensuremath{#1}}%
\settowidth{\semwidth}{\usebox{\sembox}}%
\sbox{\sembox}{\ensuremath{\left[\usebox{\sembox}\right]}}%
\settowidth{\boxwidth}{\usebox{\sembox}}%
\addtolength{\boxwidth}{-\semwidth}%
\left[\hspace{-0.3\boxwidth}%
\usebox{\sembox}%
\hspace{-0.3\boxwidth}\right]%
}

\newsavebox{\semrbox}
\newlength{\semrwidth}
\newlength{\boxrwidth}

\newcommand{\Semr}[1]{%
\sbox{\semrbox}{\ensuremath{#1}}%
\settowidth{\semrwidth}{\usebox{\semrbox}}%
\sbox{\semrbox}{\ensuremath{\left(\usebox{\semrbox}\right)}}%
\settowidth{\boxrwidth}{\usebox{\semrbox}}%
\addtolength{\boxrwidth}{-\semrwidth}%
\left(\hspace{-0.3\boxrwidth}%
\usebox{\semrbox}%
\hspace{-0.3\boxrwidth}\right)%
}

\title[Hasse principle over global function fields]
{Hasse principle for three classes of varieties over global function fields}

\author[Tian]{Zhiyu Tian}
\address{
CNRS, Fourier Institute UMR 5582\\
100 Rue des Math\'ematiques BP74\\
38402 Saint-Martin d'H\`eres Cedex, France}
\email{zhiyu.tian@ujf-grenoble.fr}

\date{\today}

\begin{document}


\begin{abstract}
We give a geometric proof that Hasse principle holds for the following varieties defined over global function fields: smooth quadric hypersurfaces in odd characteristic, smooth cubic hypersurfaces of dimension at least $4$ in characteristic at least $7$, and smooth complete intersections of two quadrics of dimension at least $3$ in odd characteristics. 

In Appendix A we explain how to modify a previous argument of the author to prove weak approximation for cubic hypersurfaces defined over function fields of curves over algebraically closed fields of characteristic at least $7$. In Appendix B we prove some corollaries of Koll\'ar's results on the fundamental group of separably rationally connected varieties.
\end{abstract}


\maketitle

\tableofcontents

\section{Introduction}
Given a variety $X$ defined over a non-algebraically closed field $K$, 
a fundamental question is to find necessary and sufficient conditions for $X$ to have a $K$-rational point. 
When the field $K$ is a global field (i.e. a number field or the function field of a curve defined over a finite field), 
we have a natural inclusion of the set of rational points of $X$, $X(K)$, into the set of ad\`elic points of $X$, $X(\mathbb{A})$. 
A classical result of Hasse-Minkowski says that if $X$ is a smooth quadric hypersurface, 
then $X(K)$ is non-empty if and only if $X(\mathbb{A})$ is non-empty. 

We say that a smooth projective variety defined over a global field 
satisfies Hasse principle if the condition that $X(\mathbb{A})$ is non-empty implies that $X(K)$ is non-empty. 
The above-mentioned result of Hasse-Minkowski can be rephrased as saying that a smooth quadric hypersurface satisfies Hasse principle. 
A natural question is to look for other varieties which satisfy Hasse principle over global fields.

Hasse principle fails in general. For (separable) rationally connected varieties, 
Colliot-Th\'el\`ene conjectured that the Brauer-Manin obstruction is the only obstruction for Hasse principle (c.f. \cite{CTBrauerManinI} P. 233 and \cite{CTBrauerManinII} P. 174 for the case of number fields). 
Therefore it is expected that smooth Fano complete intersections of dimension at least $3$ should satisfy Hasse principle.

In this article we prove the following results.

\begin{thm}\label{thm:2}
Hasse principle holds for the following varieties:
\begin{enumerate}
\item 
All smooth quadric hypersurfaces of positive dimension defined over global function fields of odd characteristic.
\item
All smooth cubic hypersurfaces in $\PP^n, n \geq 5,$ defined over global function fields of characteristic at least $7$.
\item
All smooth complete intersection of two quadric hypersurface in $\PP^n, n \geq 5,$ defined over global function fields of odd characteristic.
\end{enumerate}
\end{thm}

During the preparation of this manuscript, Browning and Vishe posted a paper on arXiv \cite{BrowningVisheCubic}, which uses the circle method to prove Hasse principle and weak approximation for smooth cubic hypersurfaces of dimension at least $6$ defined over $\FF_q(t)$ of characteristic at least $5$.

Colliot-Th\'el\`ene (\cite{CTBrauerManinII}) proves that if the field $\FF_q$ has no cubic root of the unity and the characteristic is not $3$, then the diagonal cubic hypersurface $a_0 X_0^3+a_1 X_1^3+\ldots+a_4 X_4^3=0, a_i \in \FF_q(B)$ in $\PP^4$ satisfies Hasse principle, where $B$ is a smooth projective curve.

Colliot-Th\'el\`ene and Swinnerton-Dyer prove that the Hasse principle holds for families of cubic surfaces defined by $f+tg=0 \subset \PP^3 \times \mathbb{A}^1$ \cite{CTSDpencil}.

Also it is almost certain that the results in the 167-page-long papers (\cite{CTSD}, \cite{CTSDquadric}) of Colliot-Th\'el\`ene, Sansuc, and Swinnerton-Dyer hold over global function fields of odd characteristic, although no one volunteered to write down the details.

Once we know that the Hasse principle holds for smooth complete intersections of two quadrics of dimension at least $3$, 
it is very easy to deduce weak approximation results using a geometric argument as in \cite{CTSD}. 

\begin{thm}\label{thm:wa22}
Smooth complete intersections of two quadric hypersurfaces in $\PP^n$, 
$n \geq 5,$ defined over global function fields of odd characteristic satisfy weak approximation.
\end{thm}

The more difficult problem of weak approximation on cubic hypersurfaces and a del Pezzo surface of degree $4$ will be discussed in a subsequent paper.

Our approach is geometric in nature. 
Any variety defined over $\FF_q(B)$ corresponds (non-uniquely) to a fibration $\pi: \mcX \to B$ and rational points correspond to sections of $\pi$. 
Thus we need to establish the existence of a section under the assumption that there are formal sections everywhere locally, 
which turns the problem into finding a rational point of the moduli space of sections over $\FF_q$. 
There are geometric conditions which would guarantee the existence of a rational point. But these conditions are usually difficult to check. 
When the base field is the function field of a complex curve instead of a finite field (thus the variety is defined over the function field of a complex surface), 
this line of argument is encoded in the theory of ``rational simple connectedness" (\cite{dHS}), 
the technical core of which is to check rational connectedness of the moduli space using ``very twisting surfaces".

In some sense this paper is an application of the theory of ``rational simple connectedness" in the ``non-rationally-simply-connected" case.

The basic observation of this paper is that in order to prove the existence of a section under the assumptions of the main theorem, 
it suffices to prove a much weaker statement, 
namely, the existence of a geometrically irreducible component of the space of sections (Lemma \ref{reduction}). 
To find such a component, we use a slight variant of an argument of de Jong-Starr-He \cite{dHS}. 
This argument produces a sequence of irreducible components of spaces of sections over $\bar{\FF}_q$ 
which becomes Galois invariant when the degree is large enough. 
One subtle point in our variation of their construction is that in our case the family of lines through a general point could be reducible.
So a monodromy argument is necessary to show that we can get an irreducible component. 
We deal with these problems in Appendix \ref{sec:fundamentalgroup}, using some results of Koll\'ar on the fundamental group of rationally connected varieties \cite{KollarFundamentalGroup}. 
This sequence is called ``Abel sequence" in \cite{dHS}, as it is related to the Abel-Jacobi map to the Picard variety of the base curve. 
Here we use a different name since no Abel-Jacobi map is involved. 
There seems to be several interesting arithmetic questions related to this sequence (c.f. Questions (\ref{1}), (\ref{2}), (\ref{3})). 
These are discussed in Section \ref{sec:canonical}. 

As usual, singularities cause problems in deforming sections of the family. 
The condition that one has a formal section everywhere is used to analyze singularities. 
Combining a result of Koll\'ar, which describes a ``semistable" integral model, 
we can have some control of the singularities. This is done in Section \ref{sec:ss}. 
The argument is straightforward once the theory of ``semistable model" is established. 
However the computation is quite long. The readers who trust the author's computation can simply 
take a look at Corollaries \ref{cor:semiQ}, \ref{cor:semiC}, \ref{cor:semiQQ} and proceed to the next sections.

The main theorems are proved in Sections \ref{sec:2}, \ref{sec:3}, \ref{sec:22}. 
The case of quadric hypersurface is the simplest. 
We recommend reading this case (Section \ref{sec:2}) first to get a general idea of the proof. 
The main argument is to construct a ruled surface containing two given sections (again following an idea of \cite{dHS}). 
However in our case we cannot find a chain of lines that does the job as in \cite{dHS}. 
We have to use higher degree curves. 
As a result, we have to be very careful about places of bad reductions and the degeneration of the family of rulings of the ruled surface. 
This constitutes most of the technical argument in these sections.

In the Appendix \ref{WA} we show how to modify the argument of \cite{WACubic} to prove weak approximation of cubic hypersurfaces 
defined over function fields of curves defined over an algebraically closed field of characteristic at least $7$. 
This result is used in Section \ref{sec:3}.

\textbf{Acknowledgment:} Part of this project was carried out when the author was visiting Mathematische Institut der Universit\"at Bonn. 
I would like to thank the institute, in particular Prof. Huybrechts and Mrs. Sachinidis, for their hospitality. 
I would also like to thank Prof. J.-L. Colliot-Th\'el\`ene for providing a list of references with many detailed comments, Prof. Jason Starr and Janos Koll\'ar for helpful discussions on the fundamental group of rationally connected fibrations.

\section{Semistable models over global function fields}\label{sec:ss}
In this section we review the theory of semistable models of hypersurfaces defined over global function fields by Koll\'ar \cite{KollarIntegralPolynomial} 
and generalize this to the case of complete intersections of two quadrics. 

Let $S$ be a discrete valuation ring with $K$ the quotient field and let $t$ be a generator of the maximal ideal and $k=R/(t)$ the residue field. Given a polynomial $f \in S[X_0, \ldots, X_n]$, we write $f_k$ and $f_K$ as the image of $f$ in $k[X_0, \ldots, X_n]$ and $K[X_0, \ldots, X_n]$.
\subsection{Semistable models}
We first review the case of hypersurfaces treated in \cite{KollarIntegralPolynomial}.
\begin{defn}
A weight system $W$ on $S[X_0, \ldots, X_n]$ is an $(n+1)$-tuple of integers $W=(w_0, \ldots, w_n)$. 
We write $F(W \cdot X)=F(t^{w_0}X_0, \ldots, t^{w_n}X_n)$. The multiplicity of $F$ at $W$, 
denoted by $\text{mult}_{W} F$, is defined as the minimum of the exponent of $t$ appearing in the monomials of $F(W \cdot X)$.

The family of degree $d$ hypersurfaces defined by $F \in S[X_0, \ldots, X_n]$ is semistable 
if for any change of coordinates over $S$, $X_i=a_{ij}Y_j, \det(a_{i, j}) \in S^*$, and any weight system $W$, we have $\text{mult}_W F \leq \frac{d \sum w_i}{n+1}$. 
Otherwise it is called non-semistable.
\end{defn}

The following theorem is proved in \cite{KollarIntegralPolynomial}.
\begin{thm}\label{thm:ss}
Given a degree $d$ hypersurface $f \in K[X_0, \ldots, X_n]$ which defines a semi-stable hypersurface in $\PP^n_K$ for the action of $SL(n+1)$ on $\PP(H^0(\PP^n, \OO(d)))$ in the sense of geometric invariant theory (GIT), 
there is a semistable model $F \in S[X_0, \ldots, X_n]$ such that $F_K=0$ defines a hypersurface isomorphic to $f=0$. 
In particular this holds for smooth hypersurfaces.
\end{thm}

Here we briefly discuss the proof. Let $f \in K[X_0, \ldots, X_n]$ be a homogeneous polynomial of degree d. We can find a polynomial $F \in S[X_0, \ldots, X_n]$ such that $F=0$ defines a flat family of hypersurfaces of degree $d$ over $\SP S$ 
and over $K$, the hypersurface defined $F_K=0$ is isomorphic to the hypersurface defined by $f=0$.
By geometric invariant theory, there is a homogeneous $SL(n+1)$-invariant polynomial $I$ on the coefficients of degree $d$ homogeneous polynomials 
such that when we apply the function to the coefficients of $F \in S[X_0, \ldots, X_n]$, 
we get a non-zero element of $S$, which is in the ideal $(t^k)$ for some $k \geq 0$. 
In the following, we write the value of the polynomial $I$ on the coefficients of a homogeneous degree $d$ polynomial $G$ as $I(G(X_0, \ldots, X_n))$.
Assume that $F$ is not semistable with respect to a weight system $W$. Then we perform the change of coordinates $X_i=t^{w_i}Y_i$, 
\[
F(X_0, \ldots, X_n)=F( W \cdot Y)=t^{\text{mult}_W F}F'(Y_0, \ldots, Y_n).
\]
Assume the function $I$ is of homogeneous degree $r$. Lemma 4.5 of \cite{KollarIntegralPolynomial} gives the following
\[
I(F'(X_0, \ldots, X_n))=t^{r({-\text{mult}_W F}+\frac{d \sum w_i}{n+1})}I(F(X_0, \ldots, X_n)).
\]
In other word we decrease the vanishing order of $I$ by this change of coordinates if $F$ is not semistable with respect to the weight system $W$. 
By this calculation, a semistable model is nothing but a model which minimizes the vanishing order of $I$. 
Therefore we have the following observation which will be used several times.
\begin{lem}\label{2.3}
Let $F_S \in S[X_0, \ldots, X_n]$ be a semistable family. 
Assume that there is a weight system $W$ such that $\text{mult}_W F$ equals to $\frac{d \sum w_i}{n+1}$. 
Then $F'(X_0, \ldots, X_n)= t^{-\text{mult}_W F} F(W \cdot X)$ is also a semistable family.
\end{lem}

The generalization to the case of complete intersections of two quadrics is easy. 

We first introduce some notations.
We use multi-index. So the monomial $X_0^{i_0}\ldots X_n^{i_n}, \sum_{j=0}^n i_j=d$ is abbreviated as $X_I$. 
We use these monomials as a set of basis of the vector space of homogeneous polynomials of degree $d$. 
The wedge products $X_I \wedge X_J$ form a basis of the vector space $\wedge^2 H^0(\PP^n, \OO(d))$.
Let $F_S, G_S \in S[X_0, \ldots, X_n]$ be two homogeneous polynomial of degree $d$. 
The pencil spanned by them are parameterized by the Grassmanian $G(2, H^0(\PP^n, \OO(d)))$.
We expand the wedge product $F_S \wedge G_S$ in terms of the basis $X_I \wedge X_J, I \neq J$
\[
F_S \wedge G_S=\sum a_{IJ}X_I \wedge X_J.
\]
The coefficients $a_{IJ}$ are the homogeneous coordinates under the Pl\"ucker embedding of the Grassmanian $G(2, H^0(\PP^{n}, \OO(d))$ into $\PP(\wedge^2 H^0(\PP^n, \OO(d)))$.

\begin{defn}
Let $F_S, G_S \in S[X_0, \ldots, X_n]$ be two homogeneous polynomials of degree $d$ and $W$ a weight system. 
The multiplicity of the pencil $\lambda F_S+ \mu G_S$ at the weight system $W$, denoted by $\text{mult}_W(F, G)$, 
is the minimum of the exponent of $t$ in the coefficients of $X_I \wedge X_J$ 
for all the non-zero terms $X_I \wedge X_J$ of $F_S \wedge G_S$. 
This multiplicity only depends on the pencil, not on $F$ and $G$.

We say the pencil is semistable if $\text{mult}_W (F, G) \leq \frac{2d (\sum w_i)}{n+1}$. Otherwise it is called non-semistable.
\end{defn}

Note that we always have $\text{mult}_W(F, G) \geq \text{mult}_W F+\text{mult}_W G$ for any weight system $W$. 
This is a strict inequality if and only if the lowest order term of $F(W \cdot X)$ and $G(W \cdot X)$ are proportional.
\begin{thm}
Given a pencil of degree $d$ polynomials $\lambda F_S+\mu G_S$, 
if the generic fiber defined by $F_K=G_K=0$ is GIT semi-stable for the action of $SL(n+1)$ on $\PP(\wedge^2 H^0(\PP^n, \OO(d)))$, 
then there is a semistable model of the pencil.
\end{thm}

\begin{proof}
If the pencil spanned by $(F_S, G_S)$ is not semi-stable with respect to a weight system $W=(w_0, \ldots, w_n)$, 
then set $F_S'=t^{-\text{mult}_W F_S}F_S(W\cdot X)$ and $G_S'=t^{-\text{mult}_W G_S} G_S(W \cdot X)$. 

If $\text{mult}_W{(F, G)}=\text{mult}_W F+ \text{mult}_W G$, we replace $F_S$ and $G_S$ with this new pair $(F_S', G_S')$. 

If $\text{mult}_W{(F, G)}>\text{mult}_W F+ \text{mult}_W G$, 
then $F_k'$ and $G_k'$ are proportional. 
We may write
\[
F'\equiv F'_k+t^{u}H_1 \mod t^{u+1}, G'\equiv a F'_k+t^{v}H_2 \mod t^{v+1}, u, v \geq 1, a \in k,
\] 
where $H_1, H_2$ are polynomials with coefficients in $k$ and neither of them is proportional to $F_k$. 
Without loss of generality, we may assume that $u\geq v$. 
Then $\text{mult}_W(F', G')=v=\text{mult}_W{F, G}-\text{mult}_W F- \text{mult}_W G$. 
We replace the pair $(F_S, G_S)$ with $G_S''=t^{-v}(G_S'-a F_S')$ and $F_S''=F_S'$. 

Note that $F_k$ and $G_k$ are proportional if and only if the multiplicity of the pencil $(F_S, G_S)$ is at least $1$ 
with respect to the weight system $W=(0, \ldots, 0)$, or equivalently, 
the pair is not semistable with respect to the weight system $(0, \ldots, 0)$.

Now we show that this process will eventually produce a semistable model. 
The idea is the same as Koll\'ar's argument for the case of a single polynomial. 
Let $I$ be an $SL(n+1)$ invariant homogeneous polynomial of degree $s$ in the coordinates $a_{IJ}$ 
such that $I(F_S \wedge G_S)$ is non-zero in $K$ and lie in the ideal $(t^k) \subset S$ for some $k$. 
It suffices to show that in each step we decrease the number $k$.

We have
\[
F_S(W \cdot X)=t^{\text{mult}_W F_S} F_S', G_S( W \cdot X)=t^{\text{mult}_W G_S}G_S',
\]
and
\[
F_S(W \cdot X) \wedge G_S(W\cdot X)=t^{\text{mult}_W F+\text{mult}_W G} F_S' \wedge G_S'.
\]
Thus $I(F_S (W \cdot X) \wedge G_S(W\cdot X))=t^{s(\text{mult}_W F+\text{mult}_W G)} I(F_S' \wedge G_S')$ by the homogeneity. 
On the other hand, we have
\begin{align*}
&I(F_S (W \cdot X) \wedge G_S(W\cdot X))=I(t^{\frac{d \sum w_i}{n+1}}F_S(W' \cdot X) \wedge t^{\frac{d \sum w_i}{n+1}} G_S(W' \cdot X))\\
=&t^{s\frac{2d \sum w_i}{n+1}} I(F_S(W' \cdot X) \wedge G_S(W' \cdot X))=t^{s\frac{2d \sum w_i}{n+1}} I(F_S \wedge G_S)
\end{align*}
where $W'=(w_0-\frac{d \sum w_i}{n+1}, \ldots, w_n-\frac{d \sum w_i}{n+1})$ and 
$$F_S(W' \cdot X)=F(t^{w_0-\frac{d \sum w_i}{n+1}} X_0, \ldots, t^{w_n-\frac{d \sum w_i}{n+1}}X_n),$$
$$G_S(W' \cdot X)=G(t^{w_0-\frac{d \sum w_i}{n+1}} X_0, \ldots, t^{w_n-\frac{d \sum w_i}{n+1}}X_n).$$  
The second equality follows from the $SL(n+1)$ invariance of $I$. 
Combining everything, we have
\[
I(F_S' \wedge G_S')=t^{s(\frac{2d \sum w_i}{n+1}-\text{mult}_W F-\text{mult}_W G)}I(F_S \wedge G_S).
\]
In the case $\text{mult}_W{(F, G)}>\text{mult}_W F+ \text{mult}_W G$, we also have
\[
I(F_S'' \wedge G_S'')=t^{-sv}I(F_S' \wedge G_S').
\]
Thus 
\[
I(F_S'' \wedge G_S'')=t^{s(\frac{2d \sum w_i}{n+1}-\text{mult}_W (F, G))}I(F_S \wedge G_S).
\]
So the second step will decrease the exponent of $t$ if the pair is not semistable with respect to the weight system $W$.
\end{proof}

\begin{rem}
It seems that the central fiber of a semistable model defined by $F_S=G_S=0$ may not be a complete intersection, except in the case $d=2$ (c.f. Lemma \ref{ci}). 
The reason is, we can only guarantee that $F$ and $G$ are non-proportional modulo $t$. 
It seems to the author that there is a possibility that even if we start with two polynomials whose reduction modulo $t$ defines a complete intersection, 
we cannot guarantee that we can keep this condition. 
More precisely, the author do not know if it is possible for new pair $F'$ and $G'$ (or $F''$ and $G''$) 
to have a common factor (but not proportional) modulo $t$ after each step. 
Luckily in the case $d=2$, this will not happen, because a common factor, if exists, 
will have to be a linear polynomial which will de-stablize the pair as shown below. 
\end{rem}

\begin{lem}\label{ci}
If $d=2$ and $(F_S, G_S)$ is a semi-stable family of a pencil of two quadrics, 
then $F_S=G_S=0$ defines a flat projective family of complete intersection of two quadrics over $\SP S$.
\end{lem}
\begin{proof}
It suffices to show that the family $F_S=G_S=0$ has constant fiber dimension. 
Since $F_S$ and $G_S$ are families of quadric hypersurfaces, 
the intersection has constant fiber dimension if and only if $F_k$ and $G_k$ do not contain a common linear factor over $k$
and $F_k$ and $G_k$ are not propotional. 
If $F_k$ and $G_k$ has a common linear factor, which, without loss of generality, can be assumed to be $X_0$, 
then the family is not semi-stable with respect to a weight system $(1, 0, \ldots, 0)$. If $F_k$ and $G_k$ are propotional, 
the pencil is not semi-stable with respect to the weight system $(0, \ldots, 0)$.
\end{proof}

Recall that for a semistable model and for any weight system we always have the inequalities
\[
\frac{2d\sum w_i}{n+1}\geq \text{mult}_W(F, G) \geq \text{mult}_W F+\text{mult}_W G.
\]
From this we deduce the following useful lemma.

\begin{lem}\label{2.8}
Let $(F_S, G_S)$ be a semistable family of a pencil of homogeneous polynomials of degree $d$. 
Assume that there is a weight system $W$ such that 
\[
\text{mult}_W F+ \text{mult}_W G=\frac{2d \sum w_i}{n+1}. 
\]
Then 
\[
F'(Y_0, \ldots, Y_n)= t^{-\text{mult}_W F} F(W \cdot Y) 
\]
and
\[
G'(Y_0, \ldots, Y_n)= t^{-\text{mult}_W G} G(W \cdot Y)
\]
also define a semistable family.
\end{lem}

\subsection{Semistable models for quadrics}
From now on we discuss the singularities of a semistable model over $\FF_q\Sem{t}$ (in the presence of a formal section). 
All of the results in the following three sections can be proved almost without change for any local field.
\begin{lem}\label{lem:semiQ}
Let $\mcX \to \FF_q\Sem{t}$ be a semistable model of a smooth quadric hypersurface defined over $\FF_q\Semr{t}$.
Assume the characteristic is not $2$. 
\begin{enumerate}
\item \label{Q1}$\mcX$ is smooth at any $\FF_q$-rational point in the central fiber $\mcX_0$.
\item \label{Q2} 
If there is a formal section $\whts$, 
then the central fiber $\mcX_0$ is geometrically integral and the singular locus has codimension at least $2$.
\end{enumerate}
\end{lem}

\begin{proof}
(\ref{Q1}) follows directly from the semistability with respect to the weight system $(1, \ldots, 1, 0)$ (assume that the $\FF_q$-point is $[0, \ldots, 0, 1]$).

A quadric hypersurface is geometrically integral if and only if it is not a hyperplane with multiplicity $2$ or a cone over two points. 
If there is a formal section $\whts$, then $\whts(0)$ is a rational point in the central fiber and 
the total space is smooth along this point by (\ref{Q1}). Thus the central fiber is also smooth at the point. 
In particular it is geometrically reduced. If the central fiber is a cone over two points, 
then $\whts(0)$ lies in one of the two irreducible components. 
Thus there is a linear space defined over $\FF_q$ of the central fiber, 
which is impossible since it will make the family not semistable with respect to the weight system $(1, 0, \ldots, 0)$ (assuming the linear space is $X_0=0$).
\end{proof}

As an immediate corollary, we have the following
\begin{cor}\label{cor:semiQ}
\begin{enumerate}
\item Let $\mcX \to B$ be a semistable model of a conic in $\PP^2_{\FF_q(B)}$. 
Assume the characteristic is not $2$.
If the set of ad\`elic points is non-empty, then $\mcX$ is smooth and all the fibers over closed points of $B$ are smooth.
\item Let $\mcX \to B$  be a semistable model of a quadric in $\PP^n_{\FF_q(B)}, n \geq 3$. 
Assume the characteristic is not $2$.
If the set of ad\`elic points is non-empty, then every fiber over a closed point is a geometrically integral quadric hypersurface.
\end{enumerate}
\end{cor}

\subsection{Semistable models for cubics}
\begin{lem}\label{lem:nopoint}
Let $X$ be a cubic hypersurface of positive dimension defined over a finite field $\FF_q$ of odd characteristic. 
Then either $X$ has an $\FF_q$-rational point in the smooth locus or $X$ is defined by $F(X_0, X_1, X_2)=0$, 
where $F=0$ is either a union of three Galois conjugate hyperplanes or a hyperplane with multiplicity $3$.
\end{lem}

\begin{proof}
If $X$ is a geometrically integral curve, 
then it is either a smooth genus $1$ curve or a rational curve with at most $2$ points identified. 
In any case it is easy to see that there is a point in the smooth locus. 
If $X$ is redubicle over $\FF_q$, then one of the irreducible component is a line and the statement is clear. 
A plane cubic is geometrically reducible but irreducible over $\FF_q$ if and only if it is the union of three Galois conjugate lines. 
The case of a triple line is clear.

If $X$ is cone, then the statment follows from induction on dimension.

In the following assume that $n \geq 3$ and $X$ is not a cone. 
Any hypersurface of degree $d$ in $\PP^n$ over $\FF_q$ has a rational point as long as $d \leq n$. So $X$ has a rational point. 
Assume this is $[1,0,\ldots, 0]$ and it is a singular point of $X$ of multiplicity $2$. 
We may write the equation of $X$ as
\[
X_0 Q(X_1, \ldots, X_n)+C(X_1, \ldots, X_n)=0.
\]
It suffices to show that there is a point $[x_1, \ldots x_n]$ such that $Q(X_1, \ldots, X_n)\neq 0$. 
If this is the case, then the point $$[-C(X_1, \ldots, X_n),X_1 Q(X_1, \ldots, X_n), \ldots, X_nQ(X_1, \ldots, X_n)]$$ 
is a rational point in the smooth locus. 
So it suffices to show that given any quadric hypersurface $Q$ in $\PP^n$, 
there is a rational point in $\PP^n$ which is not in the quadric hypersurface. 
This is clear if the quadric is either a hyperplane with multiplicity $2$ or a union of two hyperplanes.
In the other cases, there is a rational point $x \in Q^{sm}$.
Then one can find a line defined over $\FF_q$ passing through $x$ but is not tangent to $Q$ at $x$.
This is because the space of lines through $x$ is parametrized by $\PP^{n-1}$ and those lines which are tangent to $Q$ is a codimension $1$ linear subspace.
This line intersect $Q$ at two rational points. 
So one can always find an $\FF_q$-rationl point in the line but not in $Q$.
\end{proof}

\begin{lem}\label{lem:semiC}
Let $\mcX \to \FF_q\Sem{t}$ be a semistable family of cubic hypersurfaces in $\PP^n, n \geq 5$ such that the generic fiber is smooth. 
Also assume that the characteristic is not $2$ or $3$. 
The closed fiber is either geometrically integral, 
or a cone over three lines which are Galois conjugate to each other. 
\end{lem}
\begin{proof}
By the definition of semistability, there is no hyperplane in the central fiber which is defined over $\FF_q$. 
Thus the central fiber is either the union of three conjugate hyperplanes or geometrically integral.
\end{proof}

The following two lemmas are mostly computational. 
The general idea is that semistability gives conditions on the multiplicity of the total family along a linear space defined over $\FF_q$.
If the central fiber is non-normal or a cone over a plane cubic, we can make a suitable base change so that the total family has larger multiplicity along the linear space.
Then we will be able to find a new central fiber which is less singular by the change of coordinates as in the proof of Theorem \ref{thm:ss}. 

\begin{lem}\label{lem:cone}
Use the same assumptions as in Lemma \ref{lem:semiC}. If the central fiber is a cone over an irreducible (over $\FF_q$) plane cubic, 
and there is a formal section of the family, then there is a tower of separable degree $2$ field extensions of $\FF_q\Semr{s}/\FF_q\Semr{t}$ 
such that the base change of the generic fiber  to $\FF_q\Semr{s}$ can be extended to a family over $\SP \FF_q\Sem{s}$ 
whose central fiber is normal and not a cone over a smooth plane cubic.
\end{lem}
\begin{proof}
Since the central fiber is an irreducible plane cubic over $\FF_q$, we can write the equation of the family as
\[
F(X_0, X_1, X_2)+t G(X_3, \ldots, X_n)+ tX_0Q_0+tX_1 Q_1+tX_2 Q_2+t^2(\ldots)=0,
\]
or
\[
F(X_0, X_1)+t G(X_2, \ldots, X_n)+ tX_0Q_0+tX_1 Q_1+t^2(\ldots)=0.
\]
Note that by semistability, the second case can happen only if $n=5$ (consider the weight system $(1, 1, 0, \ldots, 0)$).

First consider the case that $G(X_3, \ldots, X_n)=0$ (or $G(X_2, \ldots, X_n)=0$) has a smooth rational point. 
Without loss of generality, assume that the point is $[X_3, \ldots, X_n]=[1, 0, \ldots, 0]$ (or$[X_2, \ldots, X_n]=[1, 0, \ldots, 0]$)
and that the tangent hyperplane of this point in the hypersurface $G=0$ is given by $X_4=0$ (or $X_3=0$). 
Then make the following base change and change of variables:
\[
t=s^2, X_0=s Y_0, X_1=s Y_1, X_2=s Y_2, X_3=Y_3, X_4=s Y_4,\ldots, X_n=s Y_n,
\]
or
\[
t=s^2, X_0=s Y_0, X_1=s Y_1, X_2= Y_2, X_3=s Y_3, \ldots, X_n=s Y_n.
\]
The new family is 
\[
F(Y_0, Y_1, Y_2)+Y_3^2 Y_4+L(Y_0, Y_1, Y_2) Y_3^2+s(\ldots)=0,
\]
or
\[
F(Y_0, Y_1)+Y_2^2 Y_3+L(Y_0, Y_1) Y_2^2+s(\ldots)=0.
\]
The central fiber $F(Y_0, Y_1, Y_2)+Y_3^2 (Y_4+L(Y_0, Y_1, Y_2))=0$ or $F(Y_0, Y_1)+Y_2^2 (Y_3+L(Y_0, Y_1))=0$ 
defines a normal cubic hypersurface which is not a cone over a plane cubic.

Then consider the case that $F(X_0, X_1, X_2)=0$ has a smooth rational point and $G(X_3, \ldots X_n)=0$ does not have a smooth rational point. Note that $F(X_0, X_1)$ cannot have a smooth rational point otherwise the family is not semistable.
By the semi-stability condition, the equation $G=0$ defines a hyperlane of multiplicity $3$ only if $n=5$ (consider the weight system $(1, 1, 1, 1, 0, \ldots, 0)$). 
In this case we can write the equation as
\[
F(X_0, X_1, X_2)+t X_3^3+tM(X_0, X_1, X_2; X_3, X_4, X_5)+t^2 H(X_4, X_5)+t^2(\ldots)=0,
\] 
where each monomial in $M(X_0, X_1, X_2; X_3, X_4, X_5)$ contains a factor of $X_0, X_1$ or $X_2$ and a factor of $X_3, X_4$ or $X_5$. 
By the semistability condition and the assumption that $F=0$ has a smooth rational point, 
the curve defined by $F(X_0, X_1, X_2)=0$ is a geometrically irreducible plane cubic.
If there are monomials of the form $X_i X_j^2$ in $M$, for some $0 \leq i \leq 2, 4 \leq j \leq 5$, then make the following change of variables
\[
t=s^2, X_0=s Y_0, \ldots, X_3=s Y_3, X_4=Y_4, X_5=Y_5.
\]
The new family becomes
\[
F(Y_0, Y_1, Y_2)+L_4(Y_0, Y_1, Y_2)X_4^2+L_5(Y_0, Y_1, Y_2)Y_5^2+s(\ldots)=0,
\]
where $L_4$ and $L_5$ are linear polynomials and at least one of them is non-zero. 
The central fiber is geometrically reduced, geometrically irreducible, normal and not a cone over a plane cubic.

Assume that there are no monomials of the form $X_i X_j^2$ in $M$, where $0 \leq i \leq 2, 4 \leq j \leq 5$. 
Make the following change of variables
\[
X_0=tY_0, \ldots, X_3=tY_3, X_4=Y_4, X_5=Y_5.
\]
The new family is still semi-stable by Lemma \ref{2.3} and can be written as
\[
H(X_4, X_5)+tF(X_0, X_1, X_2)+t^2(\ldots)=0.
\]
By semistability $H=0$ defines a cone over a union of $3$ Galois conjugate points. 
This reduces to the previous case where $F=0$ has a rational point in the smooth locus.

Next consider the case $G=0$ defines a cone over three Galois conjugate points or three Galois conjugate lines. Assume $F=0$ has a smooth rational point $[1, 0, 0]$ and the tangent line at this point is $X_1=0$. 
Make the following base change and change of variables:
\[
t=s^4, X_0=s^2 Y_0, X_1=s^3 Y_1, X_2=s^3 Y_2, X_3=s Y_3, X_4=s Y_4,\ldots, X_n=s Y_n.
\]
Then the new family is 
\[
G(Y_3, \ldots, Y_n)+Y_0^2Y_1+s(\ldots)=0.
\]
The equation $G(Y_3, \ldots, Y_n)+Y_0^2Y_1=0$ defines a normal geometrically integral cubic hypersurface which is not a cone over a plane cubic.
Note that the base change can be factorized as two degree $2$ base changes.

Finally consider the case neither $F$ or $G$ has a rational point in the smooth locus. 
Assume the formal section intersect the central fiber at $[0, \ldots, 0, 1]$. 
By lemma \ref{lem:nopoint}, $F=0$ and $G=0$ has multiplicity $3$ along the rational point. 
The semis-stability condition requires that the total family has multiplicity strictly less than $3$ along any rational point in the central fiber. 
Thus there has to exist $t^2 X_n^3$ or $t L(X_0, X_1, X_2)X_n^2$. 
By the existence of a section intersecting the central fiber at $[0, \ldots, 0, 1]$, 
the term $tL(X_0, X_1, X_2)X_n^2$ has to exist. Without loss of generality, assume $L(X_0, X_1, X_2)=X_0$. 
Then make the following change of variables:
\[
t=s^2, X_0=s Y_0, X_1=s Y_1, X_2=s Y_2, \ldots, X_{n-1}=s Y_{n-1}, X_n=Y_n.
\]
The new family is defined by 
\[
F(Y_0, Y_1, Y_2)+Y_0 Y_n^2+s(\ldots)=0.
\]
This defines a normal cubic hypersurface which is not a cone over a plane cubic 
as long as $F=0$ is not a cone over $3$ Galois conjugate points in a line. 

Note that $F=0$ is the central fiber of the original semi-stable family. 
Thus, by the semi-stability condition, it is a cone over $3$ Galois conjugate points in a line only if $n=5$ 
(and it is never a hyperplane with multiplicity $3$). 

When $n=5$, we may write the original family as
\[
F(X_0, X_1)+tG(X_2, X_3)+tM(X_0, X_1; X_2, X_3, X_4, X_5)+t^2H(X_4, X_5)+\ldots=0,
\]
or
\[
F(X_0, X_1)+tG(X_2, X_3, X_4)+tM(X_0, X_1; X_2, X_3, X_4, X_5)+a t^2 X_5^3+\ldots=0, a \neq 0
\]
where each monomial in $M(X_0, X_1; X_2, X_3, X_4, X_5)$ contains a factor of $X_0$ or $X_1$ and a factor of $X_2, X_3, X_4$ or $X_5$.

Recall that the formal section intersects the central fiber at $[0, \ldots, 0, 1]$. 
By the semistability, the multiplicity of the total family is less than $3$ at this point in the central fiber.
Thus there has to exist $\lambda t^2 X_5^3$ and $\mu t L(X_0, X_1) X_5^2$ in the defining equation. 
As before we assume $L=X_0$. As above, by the existence of a section, $\mu$ is non-zero. 
Make the following change of variables
\[
X_0=t Y_0, X_1=tY_1, X_2=Y_2, \ldots, X_5=Y_5,
\]
and the new family becomes
\[
G(Y_2, Y_3)+t(\mu Y_0 Y_5^2+ \lambda Y_5^3+\ldots)+t^2(\ldots)=0,
\]
or
\[
G(Y_2, Y_3, Y_4)+ t(\mu Y_0 Y_5^2+ \lambda Y_5^3+\ldots)+t^2(\ldots)=0.
\]
All the terms not written explicitly in $\mu X_0 X_5^2+ \lambda X_5^3+\ldots$ has a factor of $X_i, i=1, \ldots, 4$.
Then $Y_0=-\lambda, Y_5=\mu, Y_i=0, i\neq 0, 5$ is a smooth point of 
$\mu X_0 X_5^2+ \lambda X_5^3+\ldots=0 $. 
Thus we reduces to previous known cases.
\end{proof}

\begin{lem}
Use the same assumptions as in Lemma \ref{lem:semiC}. 
If the central fiber is geometrically integral but non-normal, then $n=5$. 
Either there is another semistable model whose central fiber is normal and not a cone over an irreducible plane cubic, 
or there is degree $2$ ramified base change $t=s^2$ 
such that the base change of the generic fiber to $\FF_q\Semr{s}$ can be extended to a family over $\SP \FF_q\Sem{s}$ 
whose central fiber is normal and not a cone over an irreducible plane cubic.
\end{lem}

\begin{proof}
The central fiber is non-normal if and only if its singular locus has a unique irreducible component which is a linear space of dimension $n-2$. 
To see this we pass to an algebraic closure of $\FF_q$. 
Taking $n-2$ general hyperplane sections we get an irreducible singular plane cubic curve. 
Thus there is only one singularity, which is either a node or a cusp. 
So there is only one codimension one irreducible component of the singular locus, 
which is a codimension two linear subspace in $\PP^n$ and defined over $\FF_q$. 
Note that when $n \geq 6$, a semistable model does not contain a codimension $2$ linear space defined over $\FF_q$. 
Thus the central fiber of the semistable model is non-normal only if $n$ is $5$. 

Now we work over $\FF_q$ again. Let the codimension $2$ singular locus be defined by $X_0=X_1=0$. 
Then the family can be written as
\begin{align*}
&X^2_0L_0(X_2, \ldots, X_5)+X^2_1 L_1(X_2, \ldots, X_5)+X_0 X_1 L(X_2, \ldots, X_5)+C(X_0, X_1)\\
+&tF(X_2, \ldots, X_5)+t^2(\ldots)=0
\end{align*}
where $L_0, L_1, L$ (resp. $C, F$) are linear polynomials of degree $1$ (resp. $3$) in $X_2, \ldots, X_5$.

Make the following change of variables
\[
X_0=t Y_0, X_1=t Y_1, X_2=Y_2, \ldots, X_5=Y_5.
\]
Then the new family is still semistable by Lemma \ref{2.3} and given by
\[
F(Y_2, \ldots, Y_5)+t (Y^2_0L_0+Y^2_1L_1+\ldots)+t^2(\ldots)=0.
\]

If $F(X_2, \ldots, X_5)=0$ defines a geometrically integral normal cubic hypersurface or a cone over a plane cubic, then we are either done or have reduced to previous known cases in Lemma \ref{lem:cone}.

In the following assume that $F(X_2, \ldots, X_5)=0$ defines a non-normal cubic surface in $\PP^3$ which is not a cone. 
Furthermore, assume that the singular locus is defined by $X_2=X_3=0$. 
The new family can be written as
\[
Y^2_2L_2(Y_4, Y_5)+Y^2_3 L_3(Y_4, Y_5)+Y_2 Y_3 L(Y_4, Y_5)+C'(Y_2, Y_3)+tG(Y_0, Y_1, Y_4, Y_5)+t(\ldots)=0.
\]
We can find a new semistable family whose central fiber is defined by $$G(Y_0, Y_1, Y_4, Y_5)=0$$ via a similar change of variables as above. 
So if $G(Y_0, Y_1, Y_4, Y_5)=0$ defines a normal cubic surface and not a cone over a smooth plane cubic, we are done. 
Otherwise if it defines a cone over a plane cubic, then we reduce to the case of Lemma \ref{lem:cone}.

In the following assume that $G(Y_0, Y_1, Y_4, Y_5)=0$ defines a non-normal cubic surface. 
Make the following change of variables:
\[
t=s^2, Y_2=s Z_2, Y_3=s Z_3, Y_0=Z_0, Y_1=Z_1, Y_4=Z_4, Y_5=Z_5.
\]
The new family becomes
\[
Z^2_2L_2(Z_4, Z_5)+Z^2_3 L_3(Z_4, Z_5)+Z_2 Z_3 L(Z_4, Z_5)+G(Z_0, Z_1, Z_4, Z_5)+s(\ldots)=0.
\]
We claim that this defines a normal cubic hypersurface which is not a cone over a plane cubic. 
To see this, we can compute the singular locus after making a base change to an algebraic closure of $\FF_q$. 
By the assumption that $F$ defines a non-normal cubic surface which is not a cone over a plane cubic, 
the linear span of $L_2, L_3, L$ is $2$-dimensional. Up to making linear combinations of $Z_2, Z_3$ and $Z_4, Z_5$ over $\bar{\FF}_q$, 
we may assume that either $L_2=Z_4, L_3=Z_5, L=0$ or $L_2=Z_4, L_3=0, L=Z_5$. So it suffices to show that the singular locus of
\[
Z^2_2Z_4+Z^2_3 Z_5+G(Z_0, Z_1, Z_4, Z_5)=0,
\]
and
\[
Z^2_2Z_4+Z_2 Z_3 Z_5+G(Z_0, Z_1, Z_4, Z_5)=0.
\]
has dimension at most $2$ and if there is an irreducible component of the singular locus which is isomorphic to $\PP^2$,
then the multiplicity along this $\PP^2$ is not $2$. 

In the first case the singular locus is defined by
\[
Z_2Z_4=Z_3Z_5=\frac{\partial G}{\partial Z_0}=\frac{\partial G}{\partial Z_1}=\frac{\partial G}{\partial Z_4}+Z_2^2=\frac{\partial G}{\partial Z_5}+Z_3^2=0.
\]
If $Z_2=Z_3=0$, then the singular locus is the singular locus of $G=0$, thus a codimension $4$ linear space.

If $Z_2=0, Z_3 \neq 0$, then $Z_5=0$. We also have $\frac{\partial G}{\partial Z_5}+Z_3^2=0$. Thus the singular locus has codimension at least $3$. Similarly if $Z_2\neq 0, Z_3=0$, the singular locus has codimension at least $3$.

If $Z_2 \neq 0, Z_3 \neq 0$, then $Z_4=Z_5=0$. Furthermore $\frac{\partial G}{\partial Z_4}+Z_2^2=\frac{\partial G}{\partial Z_5}+Z_3^2=0$. Thus the singular locus has codimension at least $4$.

In the second case the singular locus is defined by 
\[
2Z_2Z_4+Z_3Z_5=Z_2Z_5=\frac{\partial G}{\partial Z_0}=\frac{\partial G}{\partial Z_1}=\frac{\partial G}{\partial Z_4}+Z_2^2=\frac{\partial G}{\partial Z_5}+Z_2 Z_3=0.
\]
If $Z_2=0$, then $Z_3Z_5=0$ and $\frac{\partial G}{\partial Z_0}=\frac{\partial G}{\partial Z_1}=\frac{\partial G}{\partial Z_4}=\frac{\partial G}{\partial Z_5}=0$. The last conditions on $G$ defines a codiemension $2$ linear space. Thus the singular locus has codimension at least $3$.

If $Z_2 \neq 0$, then $Z_5=Z_4=0$ and $\frac{\partial G}{\partial Z_4}+Z_2^2=\frac{\partial G}{\partial Z_5}+Z_2 Z_3=0$. Thus the singular locus has codimension at least $3$.

In the above cases, if the singular locus has an irreducible component which is isomorphic to $\PP^{2}$, it is easy to check that the multiplicity is $2$ along this plane.
\end{proof}
We can globalize the base change and birational modifications and prove the following.
\begin{cor}\label{cor:semiC}
Let $\mcX \to B$ be a family of cubic $n$-folds ($n\geq 4$) over a smooth projective curve $B$ defined over $\FF_q$ of characteristic at least $5$. 
Assume that the set of ad\`elic points is non-empty. 
Then there is a tower of degree $2$ branched cover $C=C_1\to C_2 \to \ldots \to B$ such that the base change of the generic fiber over $B$ can be extended to a family $\mcX' \to C$ 
which has geometrically reduced, geometrically irreducible, normal fibers over closed points, 
none of which is a cone over a smooth plane cubic.
\end{cor}
\subsection{Semistable models for complete intersections of two quadrics}
This section is a straightforward computation using the theory of semistable models. First notice the following,

\begin{lem}\label{nonnormal22}
Let $X$ be a geometrically integral, non-normal complete intersection of two quadrics in $\PP^n$ defined over a field $k$ of characteristic at least $3$. 
Then the singular locus of $X$ has a unique $(n-3)$-dimensional component which is a linear space defined over $k$. 
\end{lem}

\begin{proof}
We base change to an algebraic closure $\bar{k}$ of $k$ 
and take general hyperplane sections repeatedly until the complete intersection is a reduced and irreducible curve in $\PP^3$. 
Then it is a singular curve of arithmetic genus $1$ contained in a pencil of quadric surface. 
Note that there is a smooth member of the pencil otherwise the curve is a cone over $4$ points in $\PP^2$. 
The curve has only one singular point. 
Thus there is a unique irreducible component of the singular locus which is an $(n-3)$-dimensional linear space. 
In particular this linear space is defined over $k$.
\end{proof}

\begin{lem}\label{lem:semiQQ}
Let $\mcX \to \SP \FF_q \Sem{t}$ be a semistable family of complete intersections of two quadrics defined by $Q=Q'=0$ in $\PP^n, n\geq 5$. 
Assume that the characteristic is not two and that the generic fiber of $\mcX \to B$ is smooth. 
Let $Q_0$ and $Q_0'$ be the reduction of $Q$ and $Q'$ modulo $t$.
\begin{enumerate}
\item \label{221}
None of the quadrics defined by $\lambda Q_0+\mu Q_0'=0, [\lambda, \mu] \in \PP^1(\FF_q)$ has a linear factor defined over $\FF_q$.
\item \label{222}
The closed fiber does not contain a linear subspace of dimension $n-2$ defined over $\FF_q$.
\item \label{223}
For any formal section $\whts$, at most one of $Q_0$ and $Q_0'$ is singular at the $\FF_q$-rational point $\whts(0)$.
\item \label{224}
The closed fiber is geometrically reduced.
\end{enumerate}
\end{lem}

\begin{proof}
The first two follow directly from the definition of semistable families.

For (\ref{223}), note that if a formal section of a fibration intersects the closed fiber at a singular point, 
then the total space has to be singular at this point. Thus if both $Q_0$ and $Q_0'$ are singular at $\whts(0)$, 
the two family of quadric hypersurfaces are both singular at $\whts(0)$, which violates the semistable hypothesis by looking at the multiplicity at this point.

For (\ref{224}), if the closed fiber is not geometrically reduced, 
there is a whole irreducible component which is not geometrically reduced. 
Since the closed fiber has degree $4$ and the closed fiber does not contain a linear subspace of dimension $n-2$ defined over $\FF_q$, 
the only possibilities of the closed fiber are: 
\begin{itemize}
\item 
a union of two Galois conjugate $(n-2)$-dimensional linear subspaces, each having multiplicity $2$;
\item
 a quadric of dimension $n-2$ with multiplicity $2$. 
\end{itemize}
In any case, there is a unique linear form $H$ and a quadric form $q$ such that the reduced closed fiber is defined by $H=q=0$. The quadratic polynomials $Q_0$ and $Q_0'$ are contained in the ideal generated by $H$ and $q$. So $Q_0=H \cdot L_0+ a_0 q, Q_0'=H \cdot L_0'+a_0' q$.
Then one of the quadrics $\lambda Q_0+\mu Q_0'=0, [\lambda, \mu] \in \PP^1(\FF_q)$ has a linear factor defined over $\FF_q$. 
This is impossible by (\ref{221}).
\end{proof}
\begin{lem}\label{225}
Use the same notations as in \ref{lem:semiQQ}. 
If there is a formal section of the family, then the closed fiber is either geometrically irreducible 
or there is a member of the pencil spanned by $Q_0, Q_0'$ of the form $X_0^2+aX_1^2=0$, 
where $a$ is not a square in $\FF_q$. This can happen only if $n=5$.
\end{lem}

\begin{proof}
If the closed fiber is geometrically reduced and geometrically reducible, then, 
by Lemma \ref{lem:semiQQ} (\ref{222}), the only possibilities of the closed fiber are: 
\begin{itemize}
\item
two quadrics of dimension $n-2$;
\item
a union of an $(n-2)$-dimensional quadric and two $(n-2)$-dimensional linear subspace, 
and none of the linear subspaces is defined over $\FF_q$.
\item
a union of $4$ linear subspaces of dimension $n-2$, none of which is defined over $\FF_q$;
\end{itemize}

For the first case, each of the quadrics is contained in a unique hyperplane. 
So if one of the quadrics is defined over $\FF_q$, then the corresponding hyperplane is defined over $\FF_q$. 
Then one of the quadrics $\lambda Q_0+\mu Q_0'=0, [\lambda, \mu] \in \PP^1(\FF_q)$ has a linear factor defined over $\FF_q$, which is a impossible by Lemma \ref{lem:semiQQ} (\ref{221}). 
Otherwise the two quadrics are conjugate to each other. 
So the product of the two linear forms of the corresponding hyperplanes is defined over $\FF_q$ 
and has the form $X_0^2+aX_1^2=0$ where $a$ is not a square in $\FF_q$.
Clearly this is one of quadrics in the pencil.

For the second case, the quadric is defined over $\FF_q$ by $H=q=0$. 
This is impossible by the same consideration as in the proof of Lemma \ref{lem:semiQQ} (\ref{224}).

For the last case, we first pass to an algebraic closure and take $n-3$ general linear sections. 
Then we get a reducible complete intersection in $\PP^3$, which is a union of $4$ lines. 
Furthermore two lines intersect if and only if they come from two linear subspaces intersecting at an $(n-3)$-dimensional subspace. 
Two lines are disjoint if and only if they come from two linear subspaces intersecting at an $(n-4)$-dimensional subspace. 

In $\PP^3$, the four lines are defined by a pencil of quadrics. 
First consider the case that one of the members in the pencil is a smooth quadric surface. 
In this case the lines are the rulings and the closed fiber $\mcX_0$ is a cone over the four lines. 
Two of the lines belong to one family of the ruling and the other two belong to the other family of the ruling. 
The formal section cannot intersect the closed fiber at the vertex by Lemma \ref{lem:semiQQ} (\ref{223}). 
It cannot intersect the closed fiber at the smooth locus since none of the linear spaces is defined over $\FF_q$. 
Thus it has to intersect the closed fiber at the intersection of two of the linear subspaces and not the vertex. 
Furthermore the triple intersection of the linear subspaces is empty. 
So this formal section determines two linear subspaces which intersect along a linear subspace of dimension $n-3$.
The union of these two linear subspaces is Galois invariant and is defined by $H=q=0$ over $\FF_q$.
By the same argument in the proof of Lemma \ref{lem:semiQQ} (\ref{224}), this is impossible.

If none of the member of the pencil of quadrics is smooth, then they are all cones over (possibly singular) plane conics with the same vertex. 
So the four lines intersect at the vertex, and the four $(n-2)$-dimensional linear subspace intersect at a subspace of dimension $n-3$. 
Since none of the linear subspaces is defined over $\FF_q$, the formal section must intersect the central fiber at the vertex. 
But this is impossible by Lemma \ref{lem:semiQQ} (\ref{223}).
\end{proof}

As it will become clear in the proof of Hasse principle, we need the singular fibers to be geometrically integral and not a cone over a curve in $\PP^3$. Luckily we can always find a semistable model which satisfies this requirement.

\begin{lem}\label{226}
Use the same notations as in \ref{lem:semiQQ}. 
Assume that there is a formal section. The closed fiber is non-normal or is a cone over an irreducible $(2, 2)$ complete intersection curve in $\PP^3$ only if $n=5$. 
When the closed fiber is a cone over a curve in $\PP^3$, 
there is another semi-stable model whose closed fiber is geometrically integral, non-normal but not a cone over a curve in $\PP^3$.
\end{lem}

\begin{proof}
First recall that the singular locus of a non-normal geometrically integral complete intersection of two quadrics contains a unique codimension $1$ linear space, necessarily defined over the field of definition of the complete intersection (c.f. Lemma \ref{nonnormal22}). 
If the central fiber is a cone over a geometrically irreducible curve of genus $1$, then it also contains a $(n-3)$-dimensional linear space, 
which is a cone over an $\FF_q$-rational point of the curve (the curve is either a smooth curve of genus $1$ or a rational curve and has an $\FF_q$-rational point in any case).
But by the semistability assumption, there is no $(n-3)$-dimensional linear space defined over $\FF_q$ in the closed fiber when $n\geq 6$. 
So these cases can occur only when $n$ is $5$.

Assume that the closed fiber is a cone over an irreducible curve in $\PP^3$. 
There is an $\FF_q$-point in the smooth locus of the curve. 
Up to a change of coordinates over $\FF_q\Sem{t}$, 
we may assume the point $[1, 0, 0, 0, 0, 0]$ is a rational point of the generic fiber over $\FF_q \Semr{t}$,
as well as a smooth point in the central fiber.
Assume that the two tangent hyperplanes of the two quadrics defining the central fiber are $X_1=0$ and $X_2=0$. 
Then we may write the equation as
\[
\left\{
\begin{aligned}
&X_0 X_1+q_0(X_1, X_2, X_3)+tq_1(X_4, X_5)+tX_1(\ldots)+tX_2(\ldots)+tX_3(\ldots)+t^2(\ldots)=0\\
&X_0 X_2+q_0'(X_1, X_2, X_3)+tq_1'(X_4, X_5)+tX_1(\ldots)+tX_2(\ldots)+tX_3(\ldots)+t^2(\ldots)=0
\end{aligned}
\right.
\]

We use the following change of variables
\[
X_0=Y_0, X_1=t Y_1, X_2=t Y_2, X_3=t Y_3, X_4=Y_4, X_5=Y_5.
\]
Note that both of the defining equations has multiplicity $1$ along the weight system $(0, 1, 1, 1, 0, 0)$. 
Thus the new family has to be semistable by Lemma \ref{2.8}. 
The new defining equations are
\[
\left\{
\begin{aligned}
&Y_0 Y_1+q_1(Y_4, Y_5)+t(\ldots)=0\\
&Y_0 Y_2+q_1'(Y_4, Y_5)+t(\ldots)=0.
\end{aligned}
\right.
\]
Note that $q_1$ and $q_1'$ are not proportional. 
Otherwise the pencil of quadrics contains a member which is the union of two hyperplanes defined over $\FF_q$, 
and the family cannot be semistable by Lemma \ref{lem:semiQQ} (\ref{221}). 
They cannot have a common linear factor otherwise the complete intersection contains a $3$-dimensional linear space defined over $\FF_q$, 
contradicting the semistability by Lemma \ref{lem:semiQQ} (\ref{222}). Thus $q_1(Y_4, Y_5)=q_1'(Y_4, Y_5)=0$ has no solution over $\bar{\FF}_q$. 
The new central fiber is geometrically integral, not a cone over a curve in $\PP^3$, and non-normal with singular locus $Y_0=Y_4=Y_5=0$).
\end{proof}

\begin{lem}\label{227}
Use the same notations as in \ref{lem:semiQQ}. Assume that there is a formal section. 
If there is a member $q(X_0, X_1)$ of the pencil spanned by $Q_0, Q_0'$ which defines two Galois conjugate hyperplanes, 
then there is another semistable family whose closed fiber is geometrically integral and not a cone over a curve in $\PP^3$.
\end{lem}

\begin{proof}
Recall that in this case the complete intersection is a threefold contained in $\PP^5$.
Assume the family is given by
\[
\left\{
\begin{aligned}
&q_0(X_0, X_1)+tq_1(X_2, \ldots, X_5)+tX_0(\ldots)+tX_1(\ldots)+t^2(\ldots)=0\\
&X_0L_0+X_1L_1+q_0'(X_2, \ldots, X_5)+t(\ldots)=0,
\end{aligned}
\right.
\]
where $q_0$ is irreducible over $\FF_q$. 
Up to a change of coordinates by taking linear combinations of $X_2, \ldots, X_5$, 
we may assume that the point $[0, 0, 1, 0, \ldots, 0]$ is a rational point over $\FF_q \Semr{t}$. 
It follows that there is no monomial of the form $t^k X_2^2$ in the equations above. 
By Lemma \ref{lem:semiQQ} (\ref{223}), the closed fiber of the second quadric hypersurface is smooth along $[0, 0, 1, 0, \ldots, 0]$. 
The generic fiber of two families are smooth at the point $[0, 0, 1, 0, \ldots, 0]$. 
Assume the tangent hyperplane of the two quadric hypersurfaces along the point $[0, 0, 1, 0, \ldots, 0]$ are $X_i=0$ and $X_j=0$ for some $i$ and $j$ respectively.
 Then the only monomials in the defining equations of the two quadric hypersurfaces, 
which have a factor of $X_2$, are of the form $t^{k_1} X_2X_i$ and $X_2 X_j$, $i\neq j$, respectively 
(the second equation already has this term in the zeroth order term since the central fiber is smooth along $[0, 0, 1, 0, \ldots, 0]$.  

First we claim that for a semistable family we cannot have $\{i, j\}=\{0, 1\}$. To see this we make the following change of variables
\[
X_i=tY_i, X_j=t^2 Y_j, X_2=Y_2, X_k=t Y_k, k\geq 3.
\]
The multiplicities of the two equations at the weight system $(1, 2, 0, 1, 1, 1)$ are both equal to $2$, the sum of which is exactly $\frac{4(1+2+0+1+1+1)}{6}=4$. 
Since the original family is semistable, the new family is also semistable by Lemma \ref{2.8}. 
The new family is defined by 
\[
\left\{
\begin{aligned}
&aY_i^2+t^{k_1-1}Y_2Y_i+t(\ldots)=0, a \neq 0\\
&Y_0(\ldots)+Y_1(\ldots)+Y_2 Y_j+\ldots+t(\ldots)=0
\end{aligned}
\right.
\]

But this is not semistable since the first equation has a linear factor in the zeroth order term.

So without loss of generality we may assume that $(i, j)=(3, 0), (0, 3)$ or $(3, 4)$. 
We make the following change of variables
\[
X_j=t^2 Y_j, X_2=Y_2, X_k=t Y_k, k\neq 2, j.
\]
As before the new family is semistable by Lemma \ref{2.8} and is defined by
\[
\left\{
\begin{aligned}
&aY_1^2+t^{k_1-1}Y_2Y_3+t(\ldots)=0\\
&Y_0(\ldots)+Y_1(\ldots)+Y_2 Y_0+\ldots+t(\ldots)=0
\end{aligned}
\right. \text{if }(i, j)=(3, 0)
\]
\[
\left\{
\begin{aligned}
&q_0(Y_0, Y_1)+t^{k_1-1}Y_2Y_0+t(\ldots)=0\\
&Y_0(\ldots)+Y_1(\ldots)+Y_2 Y_3+\ldots+t(\ldots)=0
\end{aligned}
\right.\text{if } (i, j)=(0, 3)
\]
\[
\left\{
\begin{aligned}
&q_0(Y_0, Y_1)+t^{k_1-1}Y_2Y_3+t(\ldots)=0\\
&Y_0(\ldots)+Y_1(\ldots)+Y_2 Y_4+\ldots+t(\ldots)=0
\end{aligned}
\right.\text{if } (i, j)=(3, 4)
\]

So $k_1=1$ and $a\neq 0$ if $(i, j)=(3, 0)$ otherwise the first equation has a linear factor in the zeroth order term. 
In the other cases, if $k_1\neq 1$, then the zeroth order term of the first defining equation is still $q(Y_0, Y_1)$ 
and we can use the same type of change of variables and produce a third semi-stable family. 
We may continue this process until we get $Y_2Y_3$ in the zeroth order term. 
So in the end, we have a new semistable family whose central fiber is defined by
\[
\left\{
\begin{aligned}
&aY_1^2+Y_2Y_3=0\\
&Y_0(\ldots)+Y_1(\ldots)+Y_2 Y_0+\ldots=0
\end{aligned}
\right. \text{if}~(i, j)=(3, 0)
\]
\[
\left\{
\begin{aligned}
&q_0(Y_0, Y_1)+Y_2Y_0=0\\
&Y_0(\ldots)+Y_1(\ldots)+Y_2 Y_3+\ldots=0
\end{aligned}
\right. \text{if}~(i, j)=(0, 3)
\]
\[
\left\{
\begin{aligned}
&q_0(Y_0, Y_1)+Y_2Y_3=0\\
&Y_0(\ldots)+Y_1(\ldots)+Y_2 Y_4+\ldots=0
\end{aligned}
\right.\text{if}~(i, j)=(3, 4)
\]
Then the central fiber is geometrically integral. 
Indeed, the new family is still semistable and has a formal section, 
so it is geometrically reduced by Lemma \ref{lem:semiQQ} (\ref{224}). 
Note that the only monomials in the defining equations above that have a factor of $Y_2$ are $Y_2Y_0, Y_2Y_3$ or $Y_2Y_4$. 
So the point $[0, 0, 1, 0, 0]$ is a smooth point of the central fiber, and in particular a smooth point of every quadric in the pencil defining the central fiber. 
Then none of the quadrics defining the central fiber is a union of two Galois conjugate hyperplanes.
By Lemma \ref{225}, the central fiber is geometrically irreducible.

If the central fiber is a cone over a $(2, 2)$-complete intersection curve in $\PP^3$, 
we can use Lemma \ref{226} to produce a new family whose fiber is geometrically reduced, geometrically irreducible, and not a cone over a curve.
\end{proof}
We may glue local semi-stable families together and prove the following.
\begin{cor}\label{cor:semiQQ}
Assume the characteristic is not $2$. 
Given a smooth complete intersection of two quadrics in $\PP^n, n\geq 5$ defined over $\FF_q(B)$.
Assuming that the set of ad\`elic points is non-empty,
then there is a semistable model over $B$ whose closed fibers are geometrically integral, 
and are not cones over a $(2, 2)$ complete intersection curve in $\PP^3$. 
The closed fibers can be non-normal only if $n=5$.
\end{cor}
\section{Asymptotically canonical sequece of spaces of sections}\label{sec:canonical}
In this section we discuss the main construction used in the proof of the main theorem, which is essentially due to \cite{dHS}.
\begin{defn}\label{def:free}
Let $\mcX \to B$ be a family of algebraic varieties defined over a field. 
A section $s: B \to \mcX$ is $m$-free if $s(B)$ is contained in the smooth locus of $\mcX$ 
and $H^1(B, \mcN_s(-b_1-\ldots-b_m))=\{0\}$ for any set of (not necessarily distinct) $m$ closed points $b_1, \ldots, b_m$ of $B$, where $\mcN_s$ is the normal sheaf of $s$ in $\mcX$. 
A morphism $f: B \to X$ is $m$-free if the induced section of the trivial family $X \times B \to B$ is free.
\end{defn}
\begin{rem}
In case of a morphism from $\PP^1$, $1$-freeness is the same as free and $2$-freeness is the same as very free. If the generic fiber is smooth projective and separably rationally connected, then the existence of a free section implies the existence of an $m$-free section for any $m>0$.
\end{rem}
We now introduce some basic hypothesis on the family $\mcX \to B$.
\begin{hypo}\label{hyp}
Given a family $\mcX \to B$ of Fano complete intersections defined over a perfect field $k$, assume the followings are satisfied.
\begin{enumerate}
\item The Fano scheme of lines of a general fiber $\mcX_b$ is smooth.
\item Choose an algebraic closure $\bar{k}$ of $k$. A general line (defined over $\bar{k}$) in a general fiber (defined over $\bar{k}$) is a free line.
\item The relative dimension of $\mcX \to B$ is at least $3$.
\item There is a free section.
\end{enumerate}
\end{hypo}

We need to introduce one more notion. Let $\mcX \to B$ be as above and $F(B) \to B$ be the relative Fano scheme of lines of the family.
The Fano scheme of lines is connected on a smooth Fano complete intersection of dimension at least $3$. 
Thus by the smoothness assumption it is irreducible for a general fiber 
and there is an open subset $U$ of the base $B$ such that the relative Fano scheme of lines $F(U)$ over $U$ is irreducible. 
Let $\bar{F} \to B$ be the closure of $F(U)$ in $F(B)$.
Denote by $\mathcal{L} \to \bar{F}$ the universal family of lines for the fibration $\mcX \to B$ restricted to the irreducible component $\bar{F}$ and let $ev_L:\mathcal{L} \to \mcX$ be the natural evaluation morphism.
It is proper and surjective.

The morphism $\mathcal{L} \to \mcX$ factors through a variety $\mathcal{Z}$ via $\mathcal{L} \to \mathcal{Z} \to \mcX$ such that a general fiber of $\mathcal{L} \to \mathcal{Z}$ is geometrically irreducible and $\mathcal{Z} \to \mathcal{X}$ is finite and generically \'etale (c.f. (9) of \cite{KollarFundamentalGroup}).
Let $\mcX^0$ be the open locus of $\mcX$ and $\mathcal{Z}^0$ be the inverse image of $\mcX^0$ in $\mathcal{Z}$ such that $\mathcal{Z}^0 \to \mathcal{X}^0$ is \'etale. 

Finally let $\mcX^1 \subset \mcX$ be the open subvariety of $\mcX$ over which the evaluation morphism $ev_L$ has constant fiber dimension.
The complement $\mcX-\mcX^1$ has codimension at least $2$ in $\mcX$.

\begin{defn}\label{def:nicesection}
Use the same notations as above. A section $s: B \to \mcX$ is a \emph{nice section} if it is $2$-free, intersects the locus $\mcX^0 \cap \mcX^1$ and the fiber product $B\times_{\mcX} \mathcal{L}$ is geometrically irreducible.
\end{defn}
It follows from the definition that a nice section is always contained in $\mcX^1$.

When the Fano index of a general fiber of $\mcX \to B$ is at least $3$, the existence of a nice section is easy. 
Since the complement of $\mcX^1$ in $\mcX$ has codimension at least $2$, a general deformation of a $2$-free section lies in $\mcX^1$. 
Then checking the irreducibility of the base change amounts to checking the irreducibility of the family of lines through a general point of the section.
We first prove the following Lemma.
\begin{lem}
Let $\mcX \to B$ be a family of Fano complete intersections defined over an algebraically closed field $\bar{k}$ which satisfy Hypothesis \ref{hyp}. Assume that the Fano index of a general fiber is at least $3$. Then a $2$-free section defined over $\bar{k}$ is nice if it contains a general $\bar{k}$-point of $\mcX$ and lies in $\mcX^1$.
\end{lem}
As indicated above, the condition of a point being ``general" can be taken to be that the family of lines through this point to be geometrically irreducible. The only non-obvious thing to check is that these points form a \emph{non-empty} open subset of $\mcX$.
\begin{proof}
Note that for any curve $T \to \mcX^1$, the every irreducible component of the base change $T\times_{\mcX}\mathcal{Z}$ dominates $T$ (c.f. proof of (7) in \cite{KollarFundamentalGroup}).
Thus it suffices to show that the family of lines through a general point of a genera fiber of $\mcX$ is geometrically irreducible. 

We look at the evaluation map $f: F_b \to \mcX_b$ of the Fano scheme of lines on a general fiber over $b \in B$.
Let $K_b$ be the function field of $\mcX_b$.
The Fano scheme $F_b$ is smooth by assumption. So the generic fiber $F_{K_b}$ of the morphism $f$ is normal.
Moreover the family of lines through a general point in $\mcX_b$ is a complete intersection of positive dimension. So $H^0(F_{K_b}, \OO_{F_{K_b}})=K_b$.
Then it is geometrically irreducible by Lemma 10, \cite{KollarFundamentalGroup}. 
Note that this is easy in characteristic zero. Only in characteristic $p$ one needs to be careful since normality is not preserved by passing to an algebraic closure.
Then a general fiber is geometrically irreducible (\cite{EGAIV} IV. 9).
\end{proof}
In general, one can still show the existence of a nice section over $\bar{k}$ for every family satisfying the Hypothesis \ref{hyp}. The proof is a simple application of the result of Koll\'ar (\cite{KollarFundamentalGroup}) on the fundamental group of separably rationally connected varieties.
For the ease of the reader, we summarize some results concerning the property of a section being ``nice" in the following. These are proved in Appendix \ref{sec:fundamentalgroup}.
\begin{lem}[=Lemma \ref{lem:familyofnicesection}]\label{lem}
Let $\mcX \to B$ be a family defined over an algebraically closed field $k$ satisfying Hypothesis \ref{hyp}. 
\begin{enumerate}
\item\label{induction1}
There is a nice section.
\item \label{general}
Let $\mathcal{S} \to W$ be an irreducible component of the space of sections such that there is a geometric point $w \in W$ which parameterizes a nice section $\mathcal{S}_w$. Then a general point of $W$ parameterizes a nice section.
\item \label{induction}
Let $\mathcal{S} \to W$ be a geometrically irreducible component of the space of sections such that a general geometric point parameterizes a nice section. Then $\mathcal{S} \times_{\mcX} \mathcal{L}$ is geometrically irreducible and generically smooth.
Furthermore it is contained in a unique geometrically irreducible component of the Kontsevich moduli space of stable sections which contains an open substack parameterizing nice sections.
\end{enumerate}
\end{lem}

Now we introduce the following hypothesis on the family $\mcX \to B$.
\begin{hypo}\label{hypnice}
Given a family $\mcX \to B$ of Fano complete intersections defined over a perfect field $k$, assume the followings are satisfied.
\begin{enumerate}
\item The Fano scheme of lines of a general fiber $\mcX_b$ is smooth.
\item Choose an algebraic closure $\bar{k}$ of $k$. A general line (defined over $\bar{k}$) in a general fiber (defined over $\bar{k}$) is a free line.
\item The relative dimension of $\mcX \to B$ is at least $3$.
\item There is a family of nice sections $\mathcal{S} \to W$ parameterized by a geometrically irreducible variety $W$ defined over $k$.
\end{enumerate}
\end{hypo}
When the field $k$ is algebraically closed, this is equivalent to Hypothesis \ref{hyp}.

Given a family $\mcX \to B$ defined over a perfect field $k$ and satisfying Hypothesis \ref{hypnice}, we have the following construction due to de Jong, He and Starr \cite{dHS}.
\begin{constr}\label{Sequence}
Start with the family $\mcX \to B$ and the family of nice section $\mathcal{S} \to W$, 
we will define a sequence of irreducible components $M_i(W), i=0, 1, \ldots$ of the moduli space of sections and their compactifications $\overline{M}_i(W), i=0, 1, \ldots$ as follows.

Define $\overline{M}_0(W)$ to be the unique irreducible component obtained as the Zariski closure of $W$ in the Kontsevich moduli space
and $M_0(W)$ to be the Zariski dense open substack of $\overline{M}_{0}(s)$ parameterizing sections. 
Denote by $\mathcal{S}_0 \to {M}_0$ the universal family of sections.

Then we define $\overline{M}_1(W)$ to be the unique geometrically irreducible component containing the family of stable sections $\mathcal{S}_0 \times_{\mcX} \mathcal{L}$.
A general point in $\overline{M}_1(W)$ parameterizes a nice section by (\ref{induction}) of Lemma \ref{lem}
Take $M_1(W)$ to the geometrcially irreducible open substack parameterizing nice sections. 

The fibration $\mcX \to B$ together with the family of nice sections parameterized by $M_1(W)$ also satisfies Hypothesis \ref{hypnice} by Lemma \ref{lem}. 
Then we can continue with the above construction, replacing $W$ by $M_1(W)$ etc..
This process produces a sequence of geometrically irreducible components $M_i(W) \subset \overline{M}_i(W), i=0, 1, \ldots$. 
\end{constr}

\begin{defn}
Given a family of Fano complete intersections $\mcX \to B$ defined over a perfect field ${k}$ satisfying Hypothesis \ref{hypnice}, 
we say that $\mcX \to B$ has an asymptotically canonical sequence if for any two geometrically irreducible component of the space of section $\mathcal{S}_1 \to W_1$ and $\mathcal{S}_2 \to W_2$ whose general points parameterize nice sections, 
there are numbers $N_1, N_2$ such that $M_{N_1+i}(W_1)=M_{N_2+i}(W_2)$ for all $i \geq 0$.
\end{defn}

The importance of this property is that it allows us to get geometrically irreducible component 
defined over a perfect field even though we do not know the existence of a free section over this field.

\begin{lem}\label{lem:Galois}
Let $\mcX \to B$ be a family of Fano complete intersections defined over a perfect field $k$ 
such that the base change to an algebraic closure $\bar{k}$ satisfies Hypothesis \ref{hyp} or equivalently Hypothesis \ref{hypnice}. 
Let $\mathcal{S} \to W$ be a nice section of the family defined over $\bar{k}$.
Furthermore assume that the sequence $M_i(W), i \geq 0$ over $\bar{k}$ is an asymptotically canonical canonical sequence, 
then every component $M_i(W)$ is Galois invariant, i.e. defined over $k$, for $i$ large enough.
\end{lem}

\begin{proof}
The family of nice sections $\mathcal{S} \to W$ is defined over a finite Galois field extension $k'/k$. 
Thus the sequence $M_i(W), i=0, 1, \ldots$ is defined over $k'$. 
It suffices to show that the Galois group $Gal(k'/k)$ fixes $M_i(W)$ for $i$ large. 
There are only finitely many geometrically irreducible components of nice sections $\mathcal{S}_1\to W_1, \ldots, \mathcal{S}_n \to W_n$ defined over $k'$ which are Galois conjugate to the family $\mathcal{S} \to W$. 
So for each $i>0$ the spaces of nice sections $M_i(W_j), j=1, \ldots, n$ are defined over $k'$ and are Galois conjugate to $M_i(W)$. 
Furthermore these are all the Galois orbits of $M_i(W)$.
By assumption, there is a finite number $N$ such that for all $i>N$ and for any $j=1, \ldots, n$, $M_i(W_j)=M_i(W)$. 
Thus the components $M_i(W)$ are Galois invariant, i.e. defined over $k$.
\end{proof}
There are several natural questions related to this construction.
\begin{enumerate}
\item \label{1}
Given a family of Fano complete intersections $\mcX \to B$ defined over a perfect field ${k}$ satisfying Hypothesis \ref{hypnice}, when can we find an asymptotically canonical sequence?
\item \label{2}
Given a family of Fano complete intersections $\mcX \to B$ defined over a perfect field $k$ 
such that the base change to an algebraic closure $\bar{k}$ satisfying Hypothesis \ref{hypnice},
is the sequence $M_i(W), i=0, 1, \ldots$ constructed over $\bar{k}$ Galois invariant for $i$ large enough? 
\item \label{3}
In view of Manin's conjecture on asymptotic behavior of rational points, 
one could ask, in the case when the sequence is asymptotically canonical, 
is there a limit of $[M_i]\cdot \mathbb{L}^{-K_\mcX \cdot s - i K_\mcX \cdot L}$ in the Grothendieck ring of varieties? 
If the family is defined over a finite field $\FF_q$ and has an asymptotically canonical sequence, 
is there a limit of the number $\frac{\#M_i(\FF_q)}{q^{K_\mcX \cdot s + i K_\mcX \cdot L}}$?
\end{enumerate}

For the purpose of this article, the importance of Questions (\ref{1}) and (\ref{2}) is because of the following 
(clearly an affirmative answer of (\ref{1}) and (\ref{2}) implies the existence of the component $\Sigma$.)
\begin{lem}\label{reduction}
Let $\mcX \to B$ be a family of varieties defined over $\FF_q$ 
such that the generic fiber $X$ has dimension at least $1$ is either a smooth quadric hypersurface, 
a smooth cubic hypersurface, or a complete intersection of two quadrics. 
If there is an irreducible component $\Sigma$ of the space of sections which is geometrically irreducible, 
then there is a section defined over $\FF_q$.
\end{lem}
\begin{proof}
By the Lang-Weil estimate, the variety $\Sigma$ has an $\FF_{q^n}$ point for every $n$ large enough. 
Equivalently, the generic fiber $X$ of the family $\pi: \mcX \to B$ has an $\FF_{q^n}(B)$ rational point for every $n$ large enough. 

A smooth quadric hypersurface $Q$ or a complete intersection of two quadrics defined over a field of odd characteristic has a rational point if and only if there is a rational point in some odd degree field extension 
(c.f. \cite{SpringerQuadric} for the case of quadrics and \cite{CI22} for the case of complete intersection of two quadrics). 
Thus in these two cases we are done. These are all we need for the proof of the main theorem.

We sketch an argument which proves the lemma in all characteristic. First note that if there is an $\FF_{q^{3n}}(B)$ rational point, then there is an $\FF_{q^n}(B)$ rational point. 
To see this, denote the $\FF_{q^{3n}}$-point and its Galois conjugate points by $x, y, z$ and consider the linear space spanned by these points. 
If this is a line, then it has to lie in $X$ since every quadric which intersects a line at three points has to contain the line.
Thus there is a rational point of $X$ in the line.
If this is a plane contained in the hypersurface $X$, then there are rational points contained in the plane.
If this is a plane which is not contained in $X$, then the intersection of the plane with $X$ is either a possibly singular conic or a zero cycle of degree $4$.
In any case there is a $0$-cycle of degree $3$ which is contained in the intersection and defined over $\FF_{q^n}(B)$. 
It then follows that there is a rational point over $\FF_{q^n}(B)$.
As discussed above, the variety $X$ has a rational point in $\FF_{q^{3^k}}(B)$ for some $k$ large enough.
Then there is an $\FF_{q}(B)$ point by this argument.

For smooth cubic hypersurfaces, we claim that if there is a rational point of $X$ defined over $\FF_{q^{2n}}(B)$, 
then there is a rational point defined over $\FF_{q^n}(B)$. 
As discussed at the beginning of the proof, the cubic hypersurface $X$ has a rational point in $\FF_{q^{2^k}}(B)$ for some $k$ large enough.
Then by the claim there is an $\FF_{q}(B)$ point.
To see this claim, note that given any $\FF_{q^{2n}}(B)$-point and its conjugate, 
they span a unique line $L$ defined over $\FF_{q^n}(B)$. 
Then either the line is contained in the hypersurface $X$ or it intersects $X$ at a third intersection point. 
In any case we have a rational point over $\FF_{q^n}(B)$.
\end{proof}

\begin{rem}
The existence of a geometrically irreducible component is almost a necessary condition. 
More precisely, if there is a section of $\mcX \to B$ which lies in the smooth locus of $\mcX$, 
then after adding sufficiently very free curves in general fibers (over $\bar{\FF}_q$) and their Galois conjugates,
we have a smooth point of an irreducible component of the Kontsevich moduli space.
Furthermore, a general point of this component parameterize a section of $\mcX$.

However, there is no guarantee that if there is a section, then it lies in the smooth locus of $\mcX$ in general. 
For a semistable family of quadric hypersurfaces, this is automatic by the definition of semi-stability. 
For complete intersection of two quadrics of dimension at least $3$, the existence of a section implies that all the fibers are geometrically integral and thus have a rational point in the smooth locus. 
Furthermore, we also know weak approximation holds once there is a rational point (Theorem \ref{thm:WA22}). So we can find a section which intersect every singular fiber at a smooth point, in particular lies in the smooth locus.
For semistable families of cubic hypersurfaces, closed fibers may not have a rational point in the smooth locus. But after a tower of degree $2$ base changes, we can always achieve this. However the weak approximation problem is still open.

On the other hand, if we can resolve the singularities of $\mcX \to B$ (which seems possible even though resolution of singularities is unknown in general, since the singularities are fairly explicit), then we can apply the argument to the resolution.
The irreducible component of the space of sections of the resolution gives a geometrically irreducible subvariety of the space of sections of $\mcX \to B$.
\end{rem}

\begin{cor}\label{cor:reduction}
Let $\mcX \to B$ be a family of varieties defined over $\FF_q$ 
such that the generic fiber $X$ has dimension at least $3$ and is either a smooth quadric hypersurface, 
a smooth cubic hypersurface, or a complete intersection of two quadrics. 
If the construction \ref{Sequence} gives an asymptotically canonical sequence of spaces of sections, then there is a section.
\end{cor}

Later in the proof we have to work with unions of sections and high degree curves in fibers. Thus it is necessary to know when does such a curve lie in the components $M_i(s)$ constructed above.

We first define the notion of a comb (with broken teeth).
\begin{defn}
A comb (with broken teeth) defined over an algebraically closed field $k$ is a prestable curve $C=C_0 \cup R_1 \cup \ldots \cup R_l$ together with a morphism $f: C \to X$ to a variety $X$, where $C_0$ is a smooth projective curve and $R_i, 1 \leq i \leq l$ are chains of rational curves attached to $C_0$ at distinct points. The curve $C_0$ is called the handle and each $R_i, 1 \leq i \leq l$ is a (broken) tooth.
\end{defn}
We need the following simple observation.

\begin{lem}\label{porcupine}
Let $\mcX \to B$ be a family of Fano complete intersections over an algebraically closed field and $s: B \to \mcX$ be a free section. 
Assume that $C\subset \mcX$ is a comb with handle $s(B)$ and teeth $C_i, i=1, \ldots, n$ 
which are free curves in smooth fibers $\mcX_{b_i}, i=1, \ldots n$. 
Furthermore assume that  every $C_i$ deforms to a chain of free lines in the fiber $\mcX_{b_i}$ and 
that the deformation can be parametrized by an irreducible curve. 
Then the comb $C$ lies in $\overline{M}_k(s)$ for some $k$ and corresponds to a smooth point in $\overline{M}_k(s)$. 
In particular, $\overline{M}_k(s)$ is the unique irreducible component containing the point corresponding to the comb $C$.
\end{lem}

\begin{proof}
For each point $b_i$, Define 
\begin{align*}
U_i=\{ &x \in \mcX_{b_i}|\text{ There is a chain of free lines } \mathcal{L}=L_1 \cup \ldots \cup L_k \text{ such that } x\\
&\text{is a point in } L_1 \text{ and lie in the smooth locus of } \mathcal{L}. \text{ The chain }\mathcal{L} \text{ lie in}\\
&\text{the same irreducible component of the Kontsevich space as } C_i.\}.
\end{align*}
By the assumptions, the complement of $U_i$ in $\mcX_{b_i}$ is a proper closed subset. Thus a general deformation of the free section $s(B)$ meets the fibers $\mcX_{b_i}$ in $U_i$.

Therefore there are families of curves $\mathcal{S} \to T, \mathcal{C}_i \to T, i=1, \ldots, n$ over an irreducible curve $T$  with sections $s_0: T \to \mathcal{S}, s_i: T \to \mathcal{C}_i$ and evaluation morphisms $ev_0: \mathcal{S} \to \mcX, ev_i: \mathcal{C}_i \to \mcX, i=1, \ldots, n$ such that
\begin{enumerate}
\item
The family $\mathcal{S} \to T$ is  a family of sections of $\mcX \to B$ and the families $\mathcal{C}_i \to T$ are families of curves in the fiber $\mcX_{b_i}$.
\item
The families $\mathcal{S} \to T, \mathcal{C}_i \to T$ glues together along the sections and form a family of combs $\mathcal{C} \to T$.
\item
There are two points $t_1, t_2$ in $T$ such that $\mathcal{C}_{t_1}$ is the comb $C$ and $\mathcal{C}_{t_2}$ is  a comb whose handle is a free section $s'(B)$ and whose teeth are chains of free lines.
\end{enumerate}

Clearly the comb $C$ and the comb $\mathcal{C}_{t_2}$ are smooth points of the \emph{same} irreducible component of the Kontsevich space of stable sections. Moreover the two sections $s(B)$ and $s'(B)$ give the same sequence of space of sections. So it suffices to show that the comb $\mathcal{C}_{t_2}$ lies in one of the $M_i(s')$'s. This follows by a simple induction on the number of lines in the teeth of the comb. 
\end{proof}

\section{Hasse principle for quadrics}\label{sec:2}
In this case we prove hasse principle for smooth quadrics defined over a global function field of odd characteristic.

First consider the case of a semistable family of smooth quadrics $\mcX \to B$ over $\FF_q$ of relative dimension at least $3$.

Still denote the base change of the family to an algebraic closure $\bar{\FF}_q$ by $\pi: {\mcX} \to {B}$. Then this family over $\bar{\FF}_q$ satisfies Hypothesis \ref{hyp}. The only non-trivial condition to check is the existence of a free section. To see this, first note that over $\bar{\FF}_q$ all the singular fibers are integral, and thus there is a point in the smooth locus. Over each singular fiber, choose a smooth point of the fiber. We know that the family ${\mcX} \to {B}$ satisfies weak approximation. Then we can find a section $s_0$ which passes through the chosen smooth points of the singular fibers. In particular, the section $s_0$ lies in the smooth locus of ${\mcX}$. Then we can add very free curves in general fibers and take a general deformation of the comb to produce a free section. Then we get nice sections by Lemma \ref{lem}.

By Corollary \ref{cor:reduction}, the existence of a section over $\FF_q$ would follow from the following lemma.

\begin{lem}\label{lem:connectQ}
Let $\pi: \mcX \to B$ be a semistable family of quadric hypersurfaces of dimension at least $3$ defined over $\bar{\FF}_q$. Then there is an asymptotically canonical sequence of sections.
\end{lem}
\begin{proof}
Let $b_1, \ldots, b_n$ be the points in $B$ whose fiber $\mcX_{b_j}$ is singular. 
Given two geometrically irreducible components of nice sections, we need to show that they produce same irreducible components when the degree of the section becomes large enough.
To show this, let $s_1$ and $s_2$ be two nice sections belonging to the two families.
Then it suffices to show that after adding enough lines in general fibers, there is a deformation parameterized by an irreducible curve of the union of $s_1$ with lines to the union of $s_2$ with lines.
The general idea to find such a deformation is to construct a ruled surface in $\mcX$ over $B$ such that $s_1$ and $s_2$ appear as sections of the ruled surface.
Below is the construction.

First of all, by taking a general deformation of $s_2$, we may assume that the line spanned by the points $s_1(b_j)$ and $s_2(b_j)$ is not contained in the fiber $\mcX_{b_j}$.

Over each point $b_j$, choose a third general point $x_j$ which does not lie in the same line with $s_1(b_j)$ or $s_2(b_j)$. By weak approximation, there is a section $s_3$ such that
\begin{enumerate}
\item
$s_3(b_j)=x_j$
\item
If $b$ is a point such that the line spanned by the two points $s_1(b)$ and $s_2(b)$ lie in the fiber $\mcX_{b}$ (the fiber $\mcX_b$ is necessarily smooth), then the line spanned by $s_3(b)$ and $s_1(b)$ (resp. $s_2(b)$) is not contained in $\mcX_b$.
\end{enumerate}

Then take the family of planes $\Pi \to B$ spanned by $s_1(b), s_2(b), s_3(b)$. The intersection of $\Pi$ with $\mcX$ is a ruled surface $S$ fibered over $B$ whose fibers are conics $\{R_b\}_{b \in B}$.

By the choice of $s_3$, the conics in the singular fibers, $R_{b_i}$, are smooth conics and $R_b$ is reduced for all $b$. Furthermore, $s_1$ and $s_2$ only intersects fibers of $S \to B$ in the smooth locus.

The surface $S$ might have singularities when the conic $R_b$ is a union of two lines and the singularity is locally given by $xy=t^n$. 
Let $\tilde{S}$ be the minimal resolution. 
Note that the fibers of $\tilde{S} \to B$ over each point is still reduced by a local computation. 
On the surface $\tilde{S}$, the strict transform of the two sections $s_1$ and $s_2$ are linearly equivalent modulo some vertical fibers. 
In fact, we have
\begin{align*}
s_1+\sum ch_j+\sum_{i=1}^k R_{b_i} \sim s_2+\sum_{i=k+1}^{n} R_{b_i},
\end{align*}
where $ch_j$ is a chain of irreducible components (each with multiplicity $1$) in the fibers 
where $s_1$ and $s_2$ are not in the same irreducible component and $R_{b_i}$ are general fibers of $\tilde{S} \to B$. 
This defines a pencil in the linear system whose general member is a smooth curve and is a section of $\tilde{S} \to B$.

Thus there is a pencil $\tilde{\mathcal{C}} \to \PP^1$ spanned by the above two divisors and a general member is a smooth curve. 
The map $\tilde{\mathcal{C}} \to \mcX$ factors through a new family $\mathcal{C} \to \mcX$ which contracts all the exceptional divisors of $\tilde{S} \to S$.

Then one can assemble two combs $C_1, C_2$ with handles $s_1$ and $s_2$ and teeth consisting of conics and lines in general fibers 
such that there is a family of stable maps $\mathcal{C}\to \PP^1$ with the following properties:
\begin{enumerate}
\item A general member of the family $\mathcal{C}_t \to \mcX$ is a section of $\mcX \to B$.
\item $\mathcal{C}_0=C_1$ and $\mathcal{C}_\infty=C_2$.
\end{enumerate}

Note that for any smooth conic in a quadric hypersurface, there is a one parameter degeneration to a union of two lines. So by Lemma \ref{porcupine}, the proof is finished.
\end{proof}
The case of conics is even easier.
\begin{lem}\label{conic_dim_one}
Let $\mcX \to B$ be a semistable model of a smooth conic in $\PP^2_{\FF_q(B)}$ which has local sections every where. Then there is a geometrically irreducible component of the space of sections.
\end{lem}

\begin{proof}
By Corollary \ref{cor:semiQ}, the fibration $\mcX \to B$ is smooth. 
Over $\bar{\FF}_q$, this family is isomorphic to $\PP(E) \to B$ for some locally free sheaf $E$ of rank $2$. 
Thus a section whose $\OO_{\PP(E)}(1)$ degree is $d$ is the same as a surjection $E \to L \to 0$ to some invertible sheaf $L$ of degree $d$. 
When $d$ is large enough, the moduli space of sections is fibered over the Jacobian of $C$ with fibers an open subset of the projective space $\PP(H^0(C, E^{*} \otimes L))$, thus irreducible. 
Obviously this component is Galois invariant and defined over $\FF_q$.
\end{proof}
This finishes the proof of the Hasse principle for conics by Lemma \ref{reduction}.

For quadric surfaces, we can reduce to the case of conics by considering the relative Fano scheme of lines. 
This is a family of conics fibered over a curve $C$ which has a generically \'etale map to $B$ of degree $2$. 
Given a rational point of a quadric surface, there are exactly two lines containing this point. 
Conversely given two lines of different families of the rulings, we get a rational point by taking the intersection of the two lines. 
So the original family has a local section everywhere if and only if the family of conics over $C$ has a section everywhere locally. 
And there is section of $\mcX \to B$ if and only if there is a section for the family of conics over $C$. 

\section{Hasse principle for cubics}\label{sec:3}
\subsection{$R$-connectedness of cubic hypersurfaces}
We first review the construction of Madore \cite{MadoreREquivalence}.

Let $X$ be a smooth cubic hypersurface of dimension at least $4$ defined over a $C_1$ field $k$ and $x, y$ two $k$-rational points in $X$. 
Then there is a chain of rational curves connecting $x$ and $y$ by a result of Madore \cite{MadoreREquivalence}. 

In the following we give a description of the chain of rational curves that connects two general points under some extra assumptions. 
This is all we need. The more general case can be treated similarly.

\begin{lem}\label{lem:rconnected}
Let $k$ be a $C_1$-field. Let $Y$ be a singular cubic hypersurface in $\PP^n$ defined over $k$ and $z$ a $k$-rational point in the singular locus of $Y$ with multiplicity $2$. Assume that the set of $k$-rational points of the projective tangent cone of $Y$ at $z$ is Zariski dense. 
Then for a general point $u$ in $Y$, there is a map $f: \PP^1 \to Y$ defined over $k$ such that $f(0)=z, f(\infty)=u$.
\end{lem}
\begin{proof}
Projection from $z$ to a hyperplane gives a birational map $Y \dashrightarrow \PP^{n-1}$.
Equivalently, one can blow up the projective space at the point $z$ and take the strict transform $\tilde{Y}$ of $Y$. 
Then there is a birational morphism $\tilde{Y} \to \PP^{n-1}$. 
The projective tangent cone of $Y$ at the point $z$ is a quadric in the exceptional divisor $E$ of $Bl_z \PP^n \to \PP^n$.
Denote by $v$ the image of $u$ in $\PP^{n-1}$ under this morphism.

Assume the coordinate of $z$ is $[1, 0, \ldots, 0]$ and the hyperplane $\PP^{n-1}$ is defined by $X_0=0$. Write the equation of $Y$ as
\[
X_0 Q(X_1, \ldots, X_n)+C(X_1, \ldots, X_n)=0.
\]
Then the projective tangent cone in $E$ is defined by $Q(X_1, \ldots, X_n)=0$.
The inverse birational map is the following:
\begin{align*}
\PP^{n-1} & \dashrightarrow Y\\
[X_1, \ldots, X_n] &\dashrightarrow [-C(X_1, \ldots, X_n), X_1Q(X_1, \ldots, X_n), \ldots, X_n Q(X_2, \ldots, X_n)]
\end{align*}

Thus the map is not defined when $Q=C=0$, which is precise the locus parameterizing the family of lines in $Y$ containing $z$.
This is also clear from the geometric description of the birational map as a projection. 
A general point in $Q=0$ is mapped to $z$. 
Choose one such general point such that the line spanned by this point and $v$ avoids the locus $Q=C=0$. 
This is possible by our assumption.
Then the rational curve we want is the restriction of the birational map to the line.
\end{proof}
\begin{rem}
The assumption is satisfied if $k$ is algebraically closed or if $k$ is the function field of a curve over an algebraically closed field and the projective tangent cone is not a union of two Galois conjugate hyperplanes.
\end{rem}

Given two general points $x, y$ in a smooth cubic hypersurface $X$ of dimension at least $4$, denote by $H_x$ and $H_y$ the tangent hyperplane of $X$ at $x$ and $y$.
Then $H_x \cap H_y \cap X$ is a smooth cubic hypersurface of dimension at least $2$. 
Thus it has a rational point $z$ over any $C_1$ field.
If $z$ is general, we can apply Lemma \ref{lem:rconnect} to $x, z$ and $y, z$ to get a chain of two rational curves.
\subsection{Geometry of cubics}
In this section we collect some useful results about the geometry of cubic hypersurfaces.
\begin{lem}\label{lem:surface}
Let $X$ be a normal cubic surface defined over an algebraically closed field $k$ of characteristic not equal to $2, 3, 5$. Assume furthermore that $X$ is not a cone over a plane cubic curve.
\begin{enumerate}
\item
There are only finitely many lines on $X$.
\item \label{surface2}
Let $x \in X$ be a general point and $H_x$ the tangent hyperplane at $x$. Then $H_x \cap X$ is a nodal cubic plane curve.
\item \label{surface3}
For any point $x \in X^{\text{sm}}$, there is a very free curve.
\item \label{surface4}
For any two general points $x, y \in X$, the intersection of their tangent hyperplanes with $X$, $H_x \cap H_y \cap X$, is smooth (i.e. $3$ distinct points).
\end{enumerate}
\end{lem}

\begin{proof}
The first statement is well-known.

For the second one, we consider the Gauss map defined on the smooth locus:
\begin{align*}
g: X &\dashrightarrow \PP^{3}\\
x \in X^{\text{sm}} &\dashrightarrow H_x
\end{align*}
The second fundamental form of the cubic surface at a smooth point is the same as the differential of the Gauss map at the point. 
Thus the tangent hyperplane intersection at a point $x$ has a node at $x$ if and only if the Gauss map is smooth (or equivalently, {\'e}tale) at the point $x$.

Denote by $Y$ the closure of the image. If we write the defining equation of $X$ as $F(X_0, \ldots, X_3)=0$, then the above map is the restriction of the map
\begin{align*}
\tilde{g}: \PP^3 &\dashrightarrow \PP^{3*}\\
[X_0, \ldots, X_3] &\dashrightarrow [\frac{\partial F}{\partial X_0}, \ldots, \frac{\partial F}{\partial X_3}].
\end{align*}
If the characteristic is not $3$ and the surface $X$ is smooth, 
then the map is a well-defined morphism on $\PP^3$. 
Clearly $\tilde{g}^*\OO(1)=\OO(2)$. Thus $\deg g \cdot \deg Y=12$. 
By smoothing and degeneration,  we see that $\deg g \leq 12$ for all normal cubic surfaces which is not a cone. 

We claim that if the characteristic of the base field $k$ is not $2, 3, 5$, then the map $g$ is separable. 
Admitting this, the map $g$ is {\'e}tale at a general point since it is generically one to one. 
Take a general point $x$ which does not lie in any line on $X$ and let $H_x$ be the tangent hyperplane at the $x$. 
Then $H_x \cap X$ has a nodal singularity at $x$.
If $y \in H_x \cap X$ is another singular point of $H_x \cap X$, 
then the intersection $H_x \cap X$ is reducible and there is a line through $x$, which contradicts the choice of $X$. 
Thus $H_x \cap X$ is a nodal plane cubic and, in particular, lies in the smooth locus. 

To see the claim, note that a map is non-separable only if the degree of $g$ is divisible by the characteristic. 
So the map is separable except possibly in the case $\deg g=\text{char} k=7$ or $11$. 
In these two cases, if the map is non-separable, then the image $Y$ is a hyperplane. 
Let $[1, 0, 0, 0]$ be a smooth point not contained in any line in $X$ and write $F=X_0^2X_1+X_0Q(X_2, X_3)+C(X_1, X_2, X_3)$. 
If the image of $X$ is a hyperplane, then there are constants $\lambda_i, i=0, \ldots, 3$ such that 
\begin{equation}\label{eq:linear}
\lambda_0 \frac{\partial F}{\partial X_0}+\ldots+\lambda_3 \frac{\partial F}{\partial X_3}=0
\end{equation}
over $X$.
We may write $Q(X_2, X_3)=aX_2^2+bX_3^2$. 
If $ab \neq 0$, then we are done. At least one of $a, b$ is non-zero otherwise the point $x$ is contained in a line. Assume that $a=1, b=0$. 
This forces $\lambda_0=\lambda_1=\lambda_2=0$. 
Also the derivative $\frac{\partial F}{\partial X_3}$ cannot be identically zero otherwise $X$ is a cone with vertex $[0, 0, 0, 1]$. So $\lambda_3$ has to be zero too 
and the image of $X$ is never a hyperplane and we finish the proof of the claim.

For (\ref{surface3}), take a general point $x$ which does not lie in any line on $X$ and let $H_x$ be the tangent hyperplane at the $x$. 
Then $H_x \cap X$ is a nodal curve.
 Let $f: \PP^1 \to H_x \cap X \subset X$ be the composition of the normalization and the inclusion map. 
Then $f$ is an immersion. 
It follows that the normal sheaf defined as the quotient of $T_{\PP^1} \to f^*T_X$ is an invertible sheaf and isomorphic to $\OO(1)$ by a simple Chern class computation. Thus $f$ is very free.

To show the last statement, it suffices to show that there is a rational curve contained in the smooth locus through every point $x \in X^{\text{sm}}$. 
Once we have this curve, we can add many general tangent hyperplane sections which lie in the smooth locus. 
Then a general smoothing with the point $x$ fixed produces a very free curve at $x$.

Choose $y$ to be a general point such that $C_y=H_y\cap X$ is a nodal plane cubic contained in $X^{\text{sm}}$ 
and such that the line spanned by $x, y$ intersects the the cubic surface at a point $z$ which does not lie on any line in $X$.

There is a birational involution:
\begin{align*}
i_z : &X \dashrightarrow X\\
&p \mapsto q
\end{align*}
where $p, q, z$ are colinear. Then $i_z(C_y)$ is a rational curve in $X^{\text{sm}}$ and contains $x$.
\end{proof}

\begin{lem}\label{lem:line}
Let $X$ be a normal cubic hypersurface in $\PP^n$ which is not a cone over a smooth plane cubic. Then the family of lines through a general point has the expected dimension $n-4$.
\end{lem}

\begin{proof}
If $x \in X$ is a vertex of a cone, then family of lines through $x$ has dimension $n-2$. Thus the family of lines through a vertex has dimension at most
\[
\dim \{x \in X| x \text{ is a vertex.}\}+n-2=2n-6.
\]

Let $x =[1, 0, \ldots, 0] \in X$ be a singular point with multiplicity $2$. We may write the defining equation of $X$ as
\[
X_0 Q(X_1, \ldots, X_n)+C(X_1, \ldots, X_n)=0.
\] 
The family of lines through $x$ is defined by $Q(X_1, \ldots, X_n)=C(X_1, \ldots, X_n)=0$ in $\PP^{n-1}$ (with coordinates $[X_1, \ldots, X_n]$).
The polynomials $Q$ and $C$ have no common factor otherwise $X$ is reducible.
Thus the family of lines through $x$ has dimension $n-3$. 
Therefore the family of lines through a singular point with multiplicity $2$ has dimension at most $\dim X^{\text{sing}}+n-3 \leq 2n-6$.

If a line lies in the smooth locus of $X$, then the Fano scheme of lines is smooth at this point and has dimension $2n-6$.

So under the assumptions on $X$, the evaluation map of Fano scheme of lines has fiber dimension at most $2n-6+1-(n-1)=n-4$ at a general point.
\end{proof}

\begin{lem}\label{lem:codim2}
\begin{enumerate}
\item
Let $X$ be smooth cubic hypersurface of dimension at least $3$ over an algebraically closed field of characteristic at least $5$. Then the locus of points in $X$ through which there does not exist a free line has codimension at least $2$.
\item 
Let $X$ be normal cubic hypersurface of dimension at least $3$ over an algebraically closed field of characteristic at least $5$. Assume that $X$ is not a cone over a plane cubic. Then the locus of points in $X$ through which the family of lines has dimension more than $n-4$ has codimension at least $2$.
\end{enumerate}
\end{lem}

\begin{proof}
First assume that $\dim X=3$.
Then the evaluation map of the universal family of lines is finite surjective of degree $6$, 
thus separable and generically \'etale when the characteristic of the field is at least $5$.
So every line through a general point of a cubic $3$-fold is free.
For a smooth cubic hypersurface of dimension higher than $3$, we can take general hyperplane sections to cut it down to a $3$-fold.
A free line in the hyperplane section is also a free line in the hypersurface $X$.
So there is a free line through a general point of $X$ and the locus has codimension at least $1$. 
But if the locus of points in $X$ through which there does not exist a free line has a codimension $1$ component, 
then it is an ample divisor which will intersect a free line.
Thus the locus has to have codimension at least $2$.

For the second part, note that the evaluation map of the universal family of lines has generic fiber dimension $n-4$ and is flat in codimension $1$.
\end{proof}

\begin{lem}\label{lem:reduced}
Let $X \in \PP^n, n\geq 5$ be a normal cubic hypersurface which is not a cone over a plane cubic. 
Given two general points $x, y$, denote by $H_x$ and $H_y$ the tangent hyperplane at $x$ and $y$. 
Then the intersection $H_x \cap H_y \cap X$ is a reduced cubic hypersurface. 
If $X$ is smooth, then this is true for any two points. 
In fact, the intersection of $X$ with any two hyperplane is reduced.
\end{lem}
\begin{proof}
If we take $(n-3)$ general hyperplane sections of $X$ then we get a normal cubic surface which is not a cone. 
The normality is clear. Since the family of lines of $X$ has the expected dimension, the cubic surface constructed has only finitely many lines. So it is not a cone.
Without loss of generality, assume the hyperplanes are $X_0=X_1= \ldots=X_{n-4}=0$. 
Denote by $F(X_0, \ldots, X_n)$ the defining equation for $X$. 
Then take the family $F(tX_0, \ldots, tX_{n-4}, X_{n-3}, \ldots, X_n)=0$ in $\PP^n \times \mathbb{A}^1$.
This is an isotrivial family $\mathcal{V} \to T\cong \mathbb{A}^1$ such that the fiber over $0$ is a cone over a normal cubic surface which is not a cone over a plane cubic and a general fiber is isomorphic to $X$.

Since the condition that $H_x \cap H_y \cap X$ is a reduced cubic hypersurface is an open condition, it suffices to show this for two general points in the central fiber, which is (\ref{surface4}) of Lemma \ref{lem:surface}.

Next we discuss the case where $X$ is a smooth cubic hypersurface. 
Let $H_1$ and $H_2$ be two hyperplanes. 
Without loss of generality, assume $H_1=\{X_0=0\}$ and $H_2=\{X_1=0\}$. 
If $X \cap H_1 \cap H_2$ is non-reduced, then the defining equation of $X$ has the form
\[
X_2^2X_3+X_0 Q_0+X_1 Q_1=0,
\]
or
\[
X_2^3+X_0 Q_0+X_1 Q_1=0.
\]
Clearly $X$ is singular along $X_0=X_1=X_2=Q_1=Q_2=0$, which is non-empty.
\end{proof}

\begin{lem}\label{lem:tangentcone}
Let $X$ be a normal cubic hypersurface of dimension at least $2$ defined over an algebraically closed field of charateristic at least $7$. Assume that $X$ is not a cone over a plane cubic.
Then for a general point $x$, the projective tangent cone of $X$ at $x$ is reduced.
If furthermore $X$ is smooth, then the projective tangent cone is a smooth quadric hypersurface.
\end{lem}

\begin{proof}
The tangent hyperplane at a general point of a normal cubic surface which is not a cone intersects the cubic surface at a nodal plane cubic. 
Thus the projective tangent cone at this point is reduced. 
We can use the same degeneration as in Lemma \ref{lem:reduced} to show this for the cubic hypersurface $X$.

If the hypersurface $X$ is smooth, it suffices to show that the Gauss map is generically \'etale, or equivalently, separable. This is true since the degree of the Gauss map is $12$ (c.f. proof of Lemma \ref{lem:surface}, (\ref{surface2})) and the characteristic is not $2$ or $3$.
\end{proof}

\begin{lem}\label{lem:break3}
Let $X$ be a smooth cubic hypersurface of dimension at least $3$ and $C$ a conic or the intersection of $X$ with a tangent plane which has only one node as the singularity. Then $C$ is a free curve and there is a deformation of $C$, parameterized by an irreducible curve, to a chain of free lines.
\end{lem}


\begin{proof}
First of all, a conic and a nodal plane cubic in a smooth cubic surface is always free. 
Then the freeness in higher dimensional case can be proved by taking general hyperplane sections containing the curve $C$ and using the normal sheaf exact sequence.
In fact a nodal plane cubic is a very free curve in a smooth cubic surface and thus is also very free in $X$.

If $C$ is a nodal plane cubic, then one can take a general deformation of $C$, which yields an embedded rational curve of degree $3$, i.e. a twisted cubic. 
The twisted cubic determines a unique $\PP^3$. If the twisted cubic is general, then the $\PP^3$ is general. 

If $C$ is a conic, then one can also find a general $\PP^3$ containing (a general deformation of) $C$.

For a general $\PP^3 \subset \PP^n$, the intersection of $X$ with the $\PP^3$ is a smooth cubic surface all of whose $27$ lines are free lines in $X$. 
One can find a deformation of a conic (resp. twisted cubic) degenerate to a chain of lines in a cubic surface (e.g. one can take the degeneration in the linear system of the conic or the twisted cubic). This gives the desired deformation in $X$.
\end{proof}

\subsection{Asymptotic canonical sequence}
By Corollary \ref{cor:semiC}, there is a tower of degree $2$ base changes $C_0 \to C_1 \to \ldots \to C_n=B$ 
such that for the cubic hypersurface $X$ defined over $\FF_q(B)$, the base change $X \times_{\FF_q(B)} \FF_q(C_0)$ has an integral model over $C_0$ 
whose closed fibers are normal and not a cone over a plane cubic.

The cubic hypersurface $X$ has an $\FF_q(C_i)$-rational point if and only if $X$ has a rational point over $\FF_q(C_{i-1})$ (c.f. proof of the case of cubic hypersurfaces in Lemma \ref{reduction}). Thus the Hasse principle for $X$ will follow from the following.
\begin{lem}
Let $\mcX \to B$ be a family of cubic hypersurfaces in $\PP^n, n\geq 5$ 
defined over an algebraically closed field of characteristic greater than $5$. 
Assume that every fiber is reduced and irreducible, normal, and not a cone over a plane cubic. 
Then there is an asymptotic canonical sequence of sections.
\end{lem}

\begin{proof}
First note that there exist free sections of the family. The proof is similar as the case of quadric hypersurface fibrations. 
We choose a point in the smooth locus of every singular fiber of the family and there is a section through these points since weak approximation is true 
for the family (Theorem \ref{thm:WACubic}).
This section lies in the smooth locus of $\mcX$.
After adding enough very free rational curves in general fibers and taking a general smoothing,
we get a free section.
Then we can find families of nice sections.
So we can apply Construction \ref{Sequence}.

Given two geometrically irreducible components of sections, choose two nice sections $s_1$ and $s_2$ in each family. Let $\mathcal{H}_1 \to B$ (resp. $\mathcal{H}_2 \to B$) be the family of tangent hyperplanes of along $s_1$ (resp. $s_2$).
Let $\mathcal{Y} \to B$ be the intersection of $\mathcal{H}_1, \mathcal{H}_2$ and $\mcX$. 
Then $\mathcal{Y} \to B$ is a family of cubic hypersurfaces in $\PP^{n-2}, n-2 \geq 3$.
Up to replacing the two sections with general deformations,
we may assume the following:
\begin{enumerate}
\item \label{c1}
For every point $b \in B$ such that $\mcX_b$ is singular, $s_1(b)$ and $s_2(b)$ are general points in $\mcX_b$. That is, the conclusions of Lemmas \ref{lem:line}, \ref{lem:reduced}, \ref{lem:tangentcone} holds for $X$.
\item \label{c2}
For every point $b \in B$ such that $\mcX_b$ is smooth, there is at least a free line through $s_1(b)$ (resp. $s_2(b)$) and the family of lines through $s_1(b)$ (resp. $s_2(b)$) has dimension $n-4$. This is possible by Lemma \ref{lem:codim2}.
\item \label{c3}
For a general point $b \in B$, the projective tangent cone of $\mathcal{H}_{s_1(b)} \cap \mcX_b$ at $s_1(b)$ (resp. $s_2(b)$) is a smooth quadric hypersurface. In addition, the fiber $\mathcal{Y}_b$ is smooth for a general point $b$.
\item \label{c4}
For every point $b \in B$, the projective tangent cone of $\mcX_b$ at $s_1(b)$ (resp. $s_2(b)$) is a reduced quadric hypersurface.
\item \label{c5}
The intersection $\mathcal{Y} \cap \mcX^0 \cap \mathcal{X}^1$ is a non-empty open subset of $\mathcal{Y}$. For the definition of $\mcX^0$ and $\mcX^1$, see the paragraph before Definition \ref{def:nicesection}.
\end{enumerate}
Note that (\ref{c4}) follows from (\ref{c2}) when the fiber $\mcX_b$ is smooth since if the projective tangent cone at a point is non-reduced, then the family of lines through this point has no smooth point.

By Lemma \ref{lem:reduced} and our choices of $s_1, s_2$, every fiber of $\mathcal{Y}$ is reduced. 

By weak approximation for the family $\mathcal{Y} \to B$ (Theorem \ref{thm:WACubic}), 
given finitely many points in the smooth locus of different fibers, 
one can find a free section $s$ of $\mathcal{Y} \to B$ passing through these points,
which can be considered as a section of $\mcX$ by composing with the inclusion $\mathcal{Y} \to \mcX$. 
This section $s$, considered as a section of $\mcX$, is free provided that its $\OO_{\mcX}(1)$ degree is sufficiently large compared to the genus of $B$.
This condition can be achieved by adding very free curves in general fibers of $\mathcal{Y} \to B$ to $s$ and taking a general deformation of the comb in $\mathcal{Y}$.
In other words, given finitely many points in the smooth locus of different fibers of $\mathcal{Y} \to B$,
there is a free section $s$ of $\mathcal{Y} \to  B$ passing through these points.
Furthermore, when considered as a section of $\mcX \to B$, this section $s$ is also free.

There are two families of projective spaces $\mathcal{P}_1=\text{Proj}(E) \to B$ and $\mathcal{P}_2=\text{Proj}(E) \to B$
together with birational maps $\mathcal{P}_1 \dashrightarrow \mcX \cap \mathcal{H}_1$ and $\mathcal{P}_2 \dashrightarrow \mcX \cap \mathcal{H}_2$, 
whose restriction to every fiber is the birational map $\PP^{n-2} \dashrightarrow \mcX_b \cap \mathcal{H}_{1 b}$ and $\PP^{n-2} \dashrightarrow \mcX_b \cap \mathcal{H}_{2 b}$ discussed in Lemma \ref{lem:rconnected}.
The space of lines through $s_1(b)$ (resp. $s_2(b)$) has dimension $n-3$ by construction. 
Since the indeterminancy locus in the fiber $\mathcal{P}_{1 b}$ (resp. $\mathcal{P}_{2 b}$) corresponds to the family of lines through $s_1(b)$ (resp. $s_2(b)$), it has codimension $2$ in each fiber.
Inside the family of projective spaces, there is a family of quadric hypersurfaces $\mathcal{Q}_1 \subset \mathcal{P}_1$ (resp. $\mathcal{Q}_2 \subset \mathcal{P}_2$) 
which corresponds the family of projective tangent cones of $\mcX \to B$ along the section $s_1: B \to \mcX$ (resp. $s_2$).
By Lemma \ref{lem:tangentcone}, every fiber of $\mathcal{Q}_1$ and $\mathcal{Q}_2$ is reduced.

Choose a general $2$-free section $s_3$ of $\mathcal{Y} \to B$ such that 
\begin{enumerate}
\item
For every point $b \in B$ such that $\mcX_b$ is singular, the point $s_3(b)$ is in the locus where the birational maps $\mathcal{P}_{1 b} \dashrightarrow \mcX_b \cap \mathcal{H}_{1 b}$ and $\mathcal{P}_{2 b} \dashrightarrow \mcX_b \cap \mathcal{H}_{1 b}$ are isomorphisms.
\item
After composing with the inclusion $\mathcal{Y} \subset \mcX$, the section is a nice section of the family $\mcX \to B$.
\end{enumerate}
It is easy to find $2$-free sections of $\mathcal{Y} \to B$ whose composition with the inclusion $\mathcal{Y} \subset \mcX$ is a $2$-free section of $\mcX \to B$. Then it is a nice section by (\ref{c5}).

In the following we show how to find a ruled surface which contains the two sections $s_1$ and $s_3$. 

There is a section $\sigma_3$ of the family $\mathcal{P}_1 \to B$ whose image under the birational map $\mathcal{P}_1 \dashrightarrow \mcX \cap \mathcal{H}_1$ is the section $s_3$.
Take a section $\sigma_1$ of the family of quadrics $\mathcal{Q}_1 \subset \mathcal{P}_1 \to B$ such that for each point $b$ where $\mcX_b$ is singular, 
the line spanned by $\sigma_1(b)$ and $\sigma_3(b)$ inside $\mathcal{P}_{1 b}$ does not meet the indeterminancy locus of $\mathcal{P}_{1 b} \dashrightarrow \mcX \cap \mathcal{H}_1$. 
This is possible since the locus, which is the same as the locus parameterizing the family of lines through $s_1(b)$, has codimension at least $2$ along every fiber over $b$ by the assumptions on $s_1$.

Let $\mathcal{L}_1 \to B$ be the family of lines in $\mathcal{P}_1 \to B$ spanned by the two sections $\sigma_1$ and $\sigma_3$. There is a rational map $\mathcal{L}_1 \dashrightarrow \mcX$.
Denote by $\mathcal{S}_1 \to B$ the normalization of the closure of the image of $\mathcal{L}_1 \dashrightarrow \mcX$. There is a morphism $ev: \mathcal{S}_1 \to \mcX$. By construction,
\begin{enumerate}
\item
Each fiber of $\mathcal{S}_1 \to B$ is either smooth or is a union of two $\PP^1$'s, both with multiplicity $1$.
\item
Each fiber over $b$ is mapped to a plane cubic which is the intersection of a tangent plane at $s_1(b)$ and $\mcX_b$.
\item 
The two sections $\sigma_1$ and $\sigma_3$ are also sections of the family $\mathcal{S}_1 \to B$, which lie in the smooth locus of $\mathcal{S}_1$, and such that $ev\circ \sigma_1=s_1, ev \circ \sigma_3=s_3$.
\item
For any point $b \in B$ such that $\mcX_b$ is singular, the fiber $\mathcal{S}_{1 b}$ is mapped to an irreducible nodal plane cubic.
\end{enumerate}
Using the ruled surface $\mathcal{S}_1 \to B$, we can assemble two combs $C_1, C_3$. 
The handle of $C_1$ (resp. $C_3$) is $s_1$ (resp. $s_3$) and the teeth are free lines, conics, and nodal plane cubics. 
We may construct the two combs in such a way that none of the teeth is contained in the singular fibers of $\mcX$. 
This is because that $\mathcal{S}_1$ has irreducible fiber over any point $b \in B$ over which the fiber $\mcX_b$ is singular. 
Furthermore there is a deformation of the two combs given by a pencil in the ruled surface $\mathcal{S}_1$.

Then by Lemma \ref{lem:break3} and Corollary \ref{porcupine}, there exist numbers $N_1, N_3$ such that $M_{N_1+i}(s_1)=M_{N_3+i}(s_3)$ for all $i \geq 0$.

Similarly we have $M_{N_2+i}(s_2)=M_{N_3'+i}(s_3)$ for some $N_2, N_3'$ and all $i \geq 0$. So there is a canonical sequence.
\end{proof}


\section{Hasse principle and weak approximation for complete intersection of two quadrics}\label{sec:22}
\subsection{$R$-equivalence and weak approximation}
In this section we briefly review a construction in \cite{CTSD}. 
\begin{lem}\label{lem:rconnect}
Let $k$ be a field of odd characteristic such that the set of rational points on every smooth quadric hypersurface of positive dimension is Zariski dense once non-empty.
Let $X$ be a smooth complete intersection of two quadrics in $\PP^n, n\geq 5$ defined over the field $k$ and $x$ a general $k$-rational points of $X$. 
If $y$ is another general (for the condition of being ``general", see the proof) $k$-rational point, 
then there is a map $f: \PP^1 \to X$ defined over $k$ such that $f(0)=f(\infty)=x, f(1)=y$.
\end{lem}
\begin{proof}
Assume $x$ has coordinate $[1, 0, \ldots, 0]$ and write the equation of $X$ as
\[
\left\{
\begin{aligned}
&X_0 X_1+q(X_1, \ldots, X_n)=0\\
&X_0 X_2+q'(X_1, \ldots, X_n)=0.
\end{aligned}
\right.
\]
Consider the pencil of tangent hyperplanes $\lambda X_1+\mu X_2$ at $x$. There is exactly one member of the pencil $\lambda X_1+\mu X_2$ contains the point $y$. Denote it by $H$. The intersection of this hyperplane $H$ and $X$ is a singular $(2, 2)$-complete intersection in $\PP^{n-1}$. Projection from $x$ one gets a quadric hypersurface in $\PP^{n-2}$ defined by the equation 
\[
\lambda q(X_1, \ldots, X_n)+\mu q'(X_1, \ldots, X_n)|_{\lambda X_1+\mu X_2=0}=0.
\]

Here we take the condition that $x, y$ are general points to mean that the quadric hypersurface 
\[
\lambda q(X_1, \ldots, X_n)+\mu q'(X_1, \ldots, X_n)|_{\lambda X_1+\mu X_2=0}=0
\]
is smooth and its hyperplane section 
\[
\lambda q(0, 0, X_3 \ldots, X_n)+\mu q'(0, 0, X_3, \ldots, X_n)=0
\]
is geometrically integral. 

For any infinite field $k$ and a general $[\lambda, \mu] \in \PP^1$, 
the hypersurface 
\[
\lambda q(X_1, \ldots, X_n)+\mu q'(X_1, \ldots, X_n)|_{\lambda X_1+\mu X_2=0}=0
\]
is smooth. 

If $n=5$ and the point $x$ is general, in the sense that there is a free line through $x$, 
a general choice of $[\lambda, \mu] \in \PP^1$ gives a smooth conic 
\[
\lambda q(0, 0, X_3 \ldots, X_5)+\mu q'(0, 0, X_3, \ldots, X_5)=0.
\] 
To see this, just note that if there is a free line, then $q$ and $q'$ cannot be simultaneously singular at the same point 
otherwise they are cones with the same vertex and there is only one line through $x$ with multiplicity $4$. 

If $n \geq 6$, then we can take $(n-5)$-general hyperplane sections to cut down to the case $n=5$. Clearly for general points $x, y$, the quadric hypersurface
\[
\lambda q(X_1, \ldots, X_n)+\mu q'(X_1, \ldots, X_n)|_{\lambda X_1+\mu X_2=0}=0
\]
is smooth and its hyperplane section 
\[
\lambda q(0, 0, X_3 \ldots, X_n)+\mu q'(0, 0, X_3, \ldots, X_n)=0
\] 
is geometrically integral if the $(n-3)$-general hyperplane sections have the same properties.

For the simplicity of the discussion, we assume in the following that the tangent hyperplane at $x$ containing $y$ is given by $X_1=0$. Thus $y$ being general means that $q(0, X_2, \ldots, X_n)=0$ and $q(0, 0, X_3, \ldots, X_n)=0$ defines a smooth quadric hypersurface $Q$ in $\PP^{n-2}$ and a geometrically integral hyperplane section of $Q$. The hypersurface $Q$, by construction, is birational to the singular $(2, 2)$ complete intersection and the birational map is explicitly given as
\begin{align*}
Q &\dashrightarrow X\cap H\\
[X_2, \ldots, X_n] &\dashrightarrow [-q'(0, X_2, \ldots, X_n), X_2^2, X_2 X_3, \ldots, X_2 X_n]
\end{align*}
The generic point of the hyperplane section $X_2=0$ of $Q$ is mapped to $x=[1, 0, \ldots, 0]$. 
The map is not defined on the locus in $Q$ satisfying $q'=X_2=0$, which is the locus parameterizing lines through $x$. 
This is also clear from the geometric description of the birational map $X \cap H \dashrightarrow Q$ as a projection.

The point $y$ is mapped to a $k$-rational point in $Q$, denoted by $u$, which does not line in the hyperplane section $X_2=0$. 
Then it is straightforward to check that there is a smooth conic through the point $u$ and two general $k$-rational points in the hyperplane section $X_2=0$ which also satisfies $q'\neq 0$. 
This conic with the rational points $u, v, w$ is the rational curve we are looking for.
\end{proof}

For later reference, we note that in the proof we have proved the following.

\begin{lem}\label{lem:generalxy}
Let $X$ be a smooth complete intersection of two quadrics in $\PP^5$ defined over an algebraically closed field $k$ of odd characteristic and $x$ a point in $X$.
Assume that there is a free line through the point $x$. 
Also assume that $x$ has coordinate $[1, 0, \ldots, 0]$ and write the equation of $X$ as
\[
\left\{
\begin{aligned}
&X_0 X_1+q(X_1, \ldots, X_n)=0\\
&X_0 X_2+q'(X_1, \ldots, X_n)=0.
\end{aligned}
\right.
\]
The for a general choice of $\lambda, \mu$, the quadric hypersurface 
\[
\lambda q(X_1, \ldots, X_n)+\mu q'(X_1, \ldots, X_n)|_{\lambda X_1+\mu X_2=0}=0
\]
and its hyperplane section 
\[
\lambda q(0, 0, X_3 \ldots, X_n)+\mu q'(0, 0, X_3, \ldots, X_n)=0
\]
are smooth.
\end{lem}

Examples of fields satisfying conditions in the lemma are $\FF_q\Semr{t}, \bar{\FF}_q\Semr{t}$, finite extensions of $\QQ_p$, $\FF(B), \bar{\FF}(B)$ (where $B$ is a smooth curve), number fields. Using this construction and the fibration method, Colliot-Th\'el\`ene, Sansuc and Swinnerton-Dyer proved the following theorem. 

\begin{thm}\cite{CTSD}\label{thm:WA22}
Let $X$ be a smooth complete intersection of two quadrics in $\PP^n, n \geq 5$ defined over a global field of odd characteristic. Assume that $X$ has a rational point. Then $X$ satisfies weak approximation.
\end{thm}
This is proved in \cite{CTSD} for the case of number fields using the above construction and the fibration method (c.f. Theorem 3.10, Theorem 3.11 of \cite{CTSD}). But the proof works in this setup as well. We refer the interested readers to \cite{CTSD} for details.

\subsection{Geometry of complete intersection of two quadrics}
In this section we collect some useful facts about complete intersections of two quadrics.
\begin{lem}\label{22smooth}
Let $X$ be a smooth complete intersection of two quadrics in $\PP^n, n\geq 5$ defined over an algebraically closed field of odd characteristic. 
\begin{enumerate}
\item \label{22smooth1}
If $n=5$, then there are four lines through a general point, all of which are free. 
\item \label{22smooth2}
Given any point $x$ in $X$, the family of lines through $x$ has dimension $n-5$.
\item \label{22smooth3}
The locus $\{x \in X | \text{there is a no free line through}~x\}$ has codimension at least $2$ in $X$.
\end{enumerate}
\end{lem}
\begin{proof}
Given a smooth complete intersection of two quadric in $\PP^5$, there is a two dimensional family of lines parameterized by a variety $U$. Moreover it is easy to see that the normal bundle of every line is either $\OO(1)\oplus \OO(-1)$ or $\OO\oplus \OO$. In any case, the normal bundle has no $H^1$. So the Fano scheme $U$ is smooth. It is well-known that $U$ is connected, thus also irreducible. Moreover, the evaluation map of the universal family is dominant. 

Assume that the defining equation can be written as
\[
\left\{
\begin{aligned}
&X_0 X_1+q(X_1, \ldots, X_5)=0\\
&X_0 X_2+q'(X_1, \ldots, X_5)=0.
\end{aligned}
\right.
\]
Thus the family of lines through $[1, 0, \ldots, 0]$ is defined by 
\[
X_1=X_2=q(X_1, \ldots, X_5)=q'(X_1, \ldots, X_5)
\] 
in $\PP^4$ (with coordinates $X_1, \ldots, X_5$). If this point is general, there are only finitely many lines through this point by dimension argument. Thus this family is a zero dimensional scheme of length four. In other word, the evaluation map of the universal family has degree $4$. Since the field has odd characteristic, the evaluation map is separable. So it is generically smooth. Then the first statement follows from standard argument of deformation theory.

For (\ref{22smooth2}), we still assume that the point $x$ is $[1, 0, \ldots, 0]$ and write the defining equations as
\[
\left\{
\begin{aligned}
&X_0 X_1+q(X_1, \ldots, X_5)=0\\
&X_0 X_2+q'(X_1, \ldots, X_5)=0.
\end{aligned}
\right.
\] 
Then there is an $(n-4)$-dimensional family of lines through $x$ if and only if $q(0, 0, X_3, \ldots, X_n)$ and $q'(0, 0, X_3, \ldots, X_n)$ has a common linear factor $L(X_3, \ldots, X_5)$. Then $X$ contains the $(n-3)$-dimensional linear space $X_1=X_2=L(X_3, \ldots, X_5)$, which is impossible by the Lefschetz hyperplane theorem for Picard groups.

For (\ref{22smooth3}), by taking hyperplane sections and using the first statement, 
we know that given any complete intersection of two quadric in $\PP^n, n\geq 5$, 
there is a free line through a general point. 
Thus the locus $$\{x \in X | \text{ There is a no free line through } x.\}$$ has codimension at least $1$. 
But $X$ has Picard number $1$. 
So this locus will intersect a free line if it is codimension $1$, 
which contradicts the definition of the locus. 
Thus it has codimension at least $2$.
\end{proof}

\begin{lem}\label{22line}
Let $X$ be a complete intersection of two quadrics in $\PP^n$ defined over an algebraically closed field of odd characteristic. Assume that $X$ is reduced and irreducible. Then the dimension of the family of lines through a general point $x \in X$ is $n-5$, unless $X$ is non-normal or is a cone over a $(2, 2)$-complete intersection curve in $\PP^3$.
\end{lem}

\begin{proof}
If $X$ is smooth, the statement holds for any point by Lemma \ref{22smooth}. So in the following we assume $X$ is singular.

Let $y$ be a singular point of $X$ which is not a vertex if $X$ is a cone. 
Assume the coordinate of the point $y$ is $[1, 0, \ldots, 0]$. Then we can write the equation of $X$ as
\[
\left\{
\begin{aligned}
&q(X_1, \ldots, X_n)=0\\
&X_0 X_2+q'(X_1, \ldots, X_n)=0.
\end{aligned}
\right.
\]
Note that $q(X_1, 0, X_3, \ldots, X_n)$ and $q'(X_1, 0, X_3, \ldots, X_n)$ has no common linear factor 
otherwise the variety $X$ contains a linear space of dimension $n-2$ and thus is reducible. 
So the family of lines through a non-vertex singular point has dimension $n-4$.

The family of lines through a vertex has dimension $n-3$.

If a line lies in the smooth locus of $X$, then the normal bundle of the line has no $H^1$. 
Thus the Fano scheme is smooth at this point and has dimension $2n-8$.

If $X$ is normal and not a cone over a $(2, 2)$-complete intersection in $\PP^3$, 
then by the above computation, 
the family of lines containing a singular non-vertex point has dimension at most $\dim X^{\text{sing}}+n-4\leq n-4+n-4=2n-8$ 
and the family of lines containing a vertex has dimension at most $n-5+n-3=2n-8$. 
Thus every irreducible component of the Fano scheme has dimension $2n-8$ 
and the fiber of the evaluation map over a general point on $X$ has dimension $n-5$.
\end{proof}

Finally we need the following result of degeneration of low degree rational curves.
\begin{lem}\label{lem:break22}
Let $X$ be a smooth complete intersection of two quadrics in $\PP^n, n\geq 5$ 
defined over an algebraically closed field of odd characteristic and $x$ a general point in $X$. 
Furthermore let $C$ be either a smooth conic, or a twisted cubic, 
or a degree $4$ nodal rational curve which is the intersection of $X$ with a $\PP^3$ tangent to complete intersection $X$ at the point $x$. 
Then there is a deformation of $C$, parameterized by an irreducible curve, to a chain of free lines.
\end{lem}
\begin{proof}
It is easy to check that the curves are free. 

Taking a general deformation of $C$, we get an embedded rational curve of degree $2, 3$ or $4$, 
which is contained in a linear projective space of dimension $2, 3$ or $4$.

By taking $(n-4)$-general hyperplanes containing a general deformation of the curve $C$, we may assume that
the curve $C$ is contained in a smooth complete intersection of two quadrics $Y$ in $\PP^4$.
Furthermore, we may assume that every line in $Y$ is a free line in $X$.
So it suffices to prove that $C$ degenerate to a chain of lines in $Y$.
One can do this explicitely.
For instance, take the degeneration in a linear system.
\end{proof}

\subsection{Asymptotic canonical sequence}
In the following we prove Hasse principle. 
First of all we study non-normal complete intersection of two quadrics.

\begin{lem}\label{lem:nonnormal22}
Let $X$ be a geometrically integral, non-normal complete intersection of two quadrics in $\PP^5$ defined over a field $k$ of characteristic at least $3$. 
Assume that the unique $(n-3)$-dimensional component of the singular locus is defined by $X_0=X_1=X_2=0$ (c.f. Lemma \ref{nonnormal22}). 
Then over $\bar{k}$ up to projective isomorphism the variety $X$ (in $\PP^5$) is defined by one of the following equations.
\[
\left\{
\begin{aligned}
&X_0 X_3+X_1^2=0\\
&X_0 X_4+X_2^2=0
\end{aligned}
\right.
\]

\[
\left\{
\begin{aligned}
&X_0 X_3+X_1 X_4+X_2 X_5=0\\
&X_0L(X_1, X_2)+Q(X_1, X_2)=0
\end{aligned}
\right.
\]
\[
\left\{
\begin{aligned}
&X_0 X_3+X_1 X_4+X_2^2=0\\
&X_0L(X_1, X_2)+Q(X_1, X_2)=0
\end{aligned}
\right.
\]
\[
\left\{
\begin{aligned}
&X_0 X_3+X_1^2+X_2^2=0\\
&X_0L(X_1, X_2)+Q(X_1, X_2)=0
\end{aligned}
\right.
\]
\end{lem}
\begin{proof}
The defining equation of $X$ can be written in the form
\[
\left\{
\begin{aligned}
&X_0 L_0(X_3, X_4, X_5)+X_1 L_1(X_3, X_4, X_5)+X_2 L_2(X_3, X_4, X_5)+Q(X_0, X_1, X_2)=0\\
&X_0 L_0'(X_3, X_4, X_5)+X_1 L_1'(X_3, X_4, X_5)+X_2 L_2'(X_3, X_4, X_5)+Q'(X_0, X_1, X_2)=0.
\end{aligned}
\right.
\]

Since it is singular along $X_0=X_1=X_2=0$, the jacobian matrix at every point in $X_0=X_1=X_2=0$ is of the form
\[
\begin{pmatrix}
&L_0(X_3, X_4, X_5) &L_1(X_3, X_4, X_5) &L_2(X_3, X_4, X_5) &0 &0 &0\\
&L_0'(X_3, X_4, X_5) &L_1'(X_3, X_4, X_5) &L_2'(X_3, X_4, X_5) &0 &0 &0\\
\end{pmatrix},
\]
and has rank at most $1$.

Up to a change of variables, we may assume that $L_0(X_3, X_4, X_5)=X_3$. 

If $L_0'$ is not a multiple of $X_3$, then we may assume that it is $X_4$. 
It then follows that $L_1, L_2$ are multiples of $X_3$ and $L_1', L_2'$ are multiples of $X_4$. 
So after a change of coordinates the equations can be written as
\[
\left\{
\begin{aligned}
&X_0 X_3+Q(X_1, X_2)=0\\
&X_0 X_4+Q'(X_1, X_2)=0.
\end{aligned}
\right.
\]
Note that $Q$ and $Q'$ has no common factor otherwise $X$ is reducible. 
Over an algebraically closed field of odd characteristic, 
we may modify the equations by taking linear combinations of the two equations and a new combination of coordinates $X_3, X_4$. Then the new equation becomes:
\[
\left\{
\begin{aligned}
&X_0 X_3+X_1^2=0\\
&X_0 X_4+X_2^2=0.
\end{aligned}
\right.
\]

If $L_0'$ is a multiple of $X_3$, write it as $\lambda X_3$. It follows that $L_i'=\lambda L_i, i=1, 2$. 
So we may assume the second equation is of the form $Q'(X_0, X_1, X_2)$.  
Depending on dimension of the $\bar{k}$-span of $L_0, L_1, L_2$ and up to a linear change of coordinates, 
we may assume that the first equation is one of the followings 
(e.g. we can eliminate monomials containing $X_0$ by replacing $X_3$ with a linear combination of $X_3$ and other linear coordinates):
\[
X_0X_3+X_1 X_4+X_2 X_5=0,
\] 
\[
X_0 X_3+X_1 X_4+ X_2^2=0,
\]
\[
X_0 X_3+ X_1^2+ X_2^2=0.
\]
Up to a change of coordinates, we may assume that there is a smooth point of $X$ of the form $[1, 0, \ldots, 0]$ 
and we may write the second equation as $X_0L(X_1, X_2)+Q(X_1, X_2)=0$ while the first equation remains the same. 
\end{proof}

\begin{lem}\label{reducedfiber}
Let $X$ be a reduced and irreducible complete intersection of two quadrics defined over an algebraically closed field $k$ of odd characteristic.
Assume that $X$ is not a cone over a complete intersection curve in $\PP^3$.
Let $x$ be a general smooth point of $X$. Assume that the point $x$ is $[1, 0, \ldots, 0]$ 
and write the defining equation of $X$ as
\[
\left\{
\begin{aligned}
&X_0X_1+Q(X_1, \ldots, X_n)=0\\
&X_0X_2+Q'(X_1, \ldots, X_n)=0
\end{aligned}
\right.
\]
Then for a general $[\lambda, \mu]\in \PP^1(k)$, 
$$\lambda Q(0, 0, X_3, \ldots, X_n)+\mu Q'(0, 0, X_3, \ldots, X_n)=0$$ defines a reduced quadric hypersurface. 
In particular, the quadric defined by $$\lambda Q(0, X_2, X_3, \ldots, X_n)+\mu Q'(0, X_2, X_3, \ldots, X_n)=0$$ is also reduced.
\end{lem}

\begin{proof}
First consider the case that $X$ is normal and not a cone over a curve in $\PP^3$.
Then for a general point $x$, the family of lines through $x$ has dimension $n-5$. 
Thus $X_1=X_2=Q=Q'=0$ cut out a complete intersection scheme in $\PP^{n-1}$. 
In particular, $Q(0, 0, X_2, \ldots, X_n)$ and $Q'(0, 0, X_2, \ldots, X_n)$ has no common factor. 
Then a general linear combination of them defines a reduced quadric hypersurface.

Next consider that case that $X$ is non-normal, but not a cone over a curve. 
We can cut down the variety $X$ by general hyperplane sections until $X$ is a threefold in $\PP^5$. 
If we can prove the statement in this case, the general case also follows.

So now assume that $X$ is a non-normal threefold in $\PP^5$ and is not a cone. 
Lemma \ref{lem:nonnormal22} classifies the defining equations. The last one is a cone over a curve in $\PP^3$. 
So we only consider the first three cases. This property is an open condition for points. 
So it suffices to find one point satisfying the condition. 
For the first two cases, one checks easily that the point $[1, 0, \ldots, 0]$ satisfies the condition. 

Finally consider the case that $X$ is defined by
\[
\left\{
\begin{aligned}
&X_0 X_3+X_1 X_4+X_2^2=0\\
&X_0L(X_1, X_2)+Q(X_1, X_2)=0
\end{aligned}
\right.
\]
If $L(X_1, X_2)$ is not of the form $bX_1$, one directly checks that such $(\lambda, \mu)$ exists for the point $[1, 0, \ldots, 0]$.
So in the following assume $L(X_1, X_2)=X_1$. 
The quadratic polynomial $Q$ contains the monomial $X_2^2$ otherwise $X$ is reducible. 
Then by a linear change of coordinates, we may assume that $X$ is defined as
\[
\left\{
\begin{aligned}
&X_0 X_3+X_1 X_4=0\\
&X_0X_1+Q(X_1, X_2)=0
\end{aligned}
\right.
\]
We may further write it as
\[
\left\{
\begin{aligned}
&X_0 X_3+X_1 X_4=0\\
&(X_0+cX_1+dX_2)X_1+X_2^2=0
\end{aligned}
\right.
\]
Then after a change of coordinates
\[
\left\{
\begin{aligned}
&X_1 X_4+X_0 X_3-d X_2X_3=0\\
&X_1 X_0+X_2^2=0
\end{aligned}
\right.
\]
Thus if $d$ is non-zero, the point $[0, 1, 0, \ldots, 0, 0]$ is what we want. 
In the following assume $d=0$. Then $[-e^2, 1, e, 0, 0, 0], e\neq 0,$ is a smooth point. 
Make the change of coordinates $$Y_0=X_0+e^2 X_1, Y_1=X_1, Y_2=X_2-e X_1, Y_3=X_3, Y_4=X_4-e^2 X_3.$$
The new equation becomes
\[
\left\{
\begin{aligned}
&Y_1 Y_4+Y_0 Y_3=0\\
&Y_1 (Y_0+2e Y_2)+Y_2^2=0
\end{aligned}
\right.
\]
Clearly $[0, 1, 0, 0, 0, 0]$ is a point we want.
\end{proof}

The following Lemma shows the existence of a canonical sequence of for a smooth $(2, 2)$-complete intersection in $\PP^5$. 
Therefore the Hasse principle holds for such a variety.

In principle this proof should also prove the general case.
But there is a subtle technical point.
The author do not know whether the generic fiber of the family $\mathcal{E}$ appearing in the proof (i.e. the hyperplane section in Lemma \ref{lem:rconnect}) is a smooth quadric hypersurface when $n$ is at least $6$.
One can show that this is the case if every line through a general point of a general fiber is free.
However the author does not know how to show this in general.
What can be proved is that this is a quadric hypersurface with at worst one singular point.
It might also be possible to deal with this case by some careful analysis.
But the author prefers to work with smooth varieties to avoid such complexity.

There are two ways to prove the Hasse principle in general case. We may either use the standard fibration method, or note that to prove the existence of an asymptotically canonical sequence, we can take a general family of hyperplane sections to reduce to the case $n=5$ and construct the ruled surface.

\begin{lem}\label{lem:connect22}
Let $\pi: \mcX \to B$ be a family of complete intersection of two quadric hypersurfaces in $\PP^5$ 
defined over an algebraically closed field $k$ of odd characteristic. Then there is an asymptotically canonical sequence.
\end{lem}
\begin{proof}
First of all, we show the existence of a nice section.
This basically follows the same line of argument as before.
We use the fact that there are smooth points in every closed fiber and weak approximation holds for smooth complete intersection of two quadrics (e.g. by \cite{CTGilleWA}) to find a section in the smooth locus of $\mcX$.
Then the standard argument of attaching very free curves and smoothing produces a free section and then a nice section by Lemma \ref{lem}.

Given two different irreducible family of sections, we will show that the construction \ref{Sequence} eventually gives the same irreducible components of spaces of sections. Choose two nice sections $s_1$ and $s_2$ in each of these two families.
We may replace the two sections $s_1$ and $s_2$ by their general deformations in the moduli such that
\begin{enumerate}
\item
Over every point $b \in B$ where the fiber is singular, the points $s_1$ and $s_2$ are general in the sense that the conditions of Lemmas \ref{reducedfiber}, \ref{22line} hold.
\item
There are only finitely many points $b \in B$ such that there is a non-free line in $\mcX_b$ passing through $s_1(b)$.
\item
For any point $b \in B$ such that the fiber $\mcX_b$ is smooth, there is a free line through $s_1(b)$.
\end{enumerate}

Now we construct a ruled surface $\mathcal{S} \to B$ and a $B$-morphism $f:\mathcal{S} \to \mcX$ together with two sections $\sigma_1$ and $\sigma_2$ such that
\begin{enumerate}
\item
$f\circ \sigma_i=s_i, i=1, 2$.
\item
For any point $b \in B$, the points $\sigma_1(b)$ and $\sigma_2(b)$ lie in the smooth locus of the fiber.
\item 
Over any point $b \in B$, the fiber of $\mathcal{S}$ over $B$ is reduced and has at most two irreducible components.
\item
Over each $b$ such that $\mcX_b$ is singular, the fiber of $\mathcal{S}$ over $b$ is either smooth or 
a union of two $\PP^1$'s such that $\sigma_1$ and $\sigma_2$ lie in the smooth locus of the same irreducible component.
\item
A general fiber of $\mathcal{S}$ over $b \in B$ is mapped to the intersection of $\mcX_b$ with a tangent $\PP^3$ at $s_1(b)=f\circ \sigma_1(b)$. 
Each irreducible component of reducible fibers of $S$ are mapped to embedded rational curves of degree $1, 2$ or $3$.
\end{enumerate}
The construction of such a surface uses Lemma \ref{lem:rconnect}. 
First of all, one takes the family of tangent hyperplanes along $s_1$ which contains the section $s_2$ 
and then project the hyperplane section to a family of quadric surfaces, as described in Lemma \ref{lem:rconnect}. 
The generic fiber is a smooth quadric surface $Q$ defined over $k(B)$. 
Over the generic fiber, there is a hyperplane section $E$ of this quadric whose image under the birational map is the rational point corresponding to $s_1$. 
This hyperplane section $E$ is smooth over $k(B)$ by Lemma \ref{lem:generalxy}.
 
Denote by $\mathcal{Q} \to B$ and $\mathcal{E} \to B$ the family of quadric hypersurfaces and its hyperplane sections.
There is a rational map $\mathcal{Q} \dashrightarrow \mcX$ which is defined away from a multi-section of degree $4$ of $\mathcal{E} \to B$.

By Lemma \ref{reducedfiber} and our choice of $s_1$ and $s_2$, both the family of quadric surfaces and the family of hyperplane sections have reduced fibers over every point $b \in B$. 

Since smooth quadrics fibrations of positive dimension over a curve $B$ satisfies weak approximation, 
and the families $\mathcal{Q} \to B$ and $\mathcal{E} \to B$ have reduced fibers over every point, 
we may find three sections $\tilde{\sigma}_1, \tilde{\sigma}_2, \tilde{\sigma}_3$ of the family $\mathcal{Q} \to B$ such that
\begin{enumerate}
\item
The section $\tilde{\sigma}_2$ is contained in the locus where the rational map $\mathcal{Q} \dashrightarrow \mcX$ is defined 
and is mapped to the section $s_2$ under this birational map.
\item
The sections $\tilde{\sigma}_1$ and $\tilde{\sigma}_3$ are also sections of the family $\mathcal{E} \to B$.
\item
If the fiber of $\mathcal{Q}$ is reducible over a point $b$, 
then $\tilde{\sigma}_1(b)$ and $\tilde{\sigma}_2(b)$ lies in the same irreducible component of $\mathcal{Q}_b$.
\item
If the fiber of $\mathcal{Q}$ is singular but irreducible over a point $b$, 
then the plane spanned by $\tilde{\sigma}_1(b)$ and $\tilde{\sigma}_2(b)$ and $\tilde{\sigma}_3(b)$ intersect $\mathcal{Q}_b$ at a smooth conic.
\item
The sections $\tilde{\sigma}_1$ and $\tilde{\sigma}_3$ intersect the indeterminancy locus of $\mathcal{Q} \dashrightarrow \mcX$ transverserly
over the points $b$ where the fibers of $\mcX, \mathcal{Q}$, and $\mathcal{E}$ are smooth.
\end{enumerate} 

Let $\Pi \to B$ be the family of planes spanned by the three sections $\tilde{\sigma}_1, \tilde{\sigma}_2, \tilde{\sigma}_3$ and let $\mathcal{S}_0$ be the family of conics of the intersection of the family of planes $\Pi$ and the family $\mathcal{Q}$. 
There is a rational map $S_0 \dashrightarrow \mathcal{X}$. 
The surface $\mathcal{S}$ is constructed as the normalization of the closure of the image of $\mathcal{S}_0$. 
The sections $\sigma_1$ and $\sigma_2$ are strict transforms of $\tilde{\sigma}_1$ and $\tilde{\sigma}_2$.

As before, we can construct a pencil inside the minimum resolution of $\mathcal{S}$, 
which gives a deformation of a comb with handle $\sigma_1$ and a comb with handle $\sigma_2$.
The pencil induces a deformation of a comb with handle $s_1$ and a comb with handle $s_2$ in $\mcX$. 
By construction, we may take the teeth to be free lines, conics, twisted cubics, and a nodal rational curve tangent to $\mcX_b$ along $s_1(b)$ for some $b$.
Furthermore, these curves all lie in the smooth fibers of $\mcX \to B$. 
This is because, by construction, over a point $b \in B$ where $\mcX_b$ is singular, 
the sections $\sigma_1$ and $\sigma_2$ intersect the fiber of $\mathcal{S}$ over $b$ at the same irreducible component even if the fiber of $\mathcal{S}$ over $b$ could be reducible.
So there is no need to add curves in singular fibers of $\mcX$ to construct the comb.
Finally the existence of the canonical sequence follows from Lemma \ref{lem:break22} and Lemma \ref{porcupine}.
\end{proof}

\appendix

\section{Weak approximation for cubic hypersurfaces defined over function fields of curves}\label{WA}
In this appendix we indicate how to modify the argument of \cite{WACubic} to prove the following theorem.

\begin{thm}\label{thm:WACubic}
Let $X$ be a smooth cubic hypersurface in $\PP^n, n \geq 3$ defined over the function field $K(B)$ 
of a smooth curve $B$ over an algebraically closed field $K$ of characteristic not equal to $2, 3, 5$. 
Then $X$ satisfies weak approximation.
\end{thm}

The first result we need is the following.
\begin{thm}[\cite{CortiCubic}, \cite{KollarIntegralPolynomial}]
Let $X$ be smooth projective cubic surface in $\PP^3$ over $K(B)$, 
the function field of a curve $C$ defined over an algebraically closed field $K$ of characteristic not equal to $2, 3$. 
Then there is an integral model $\mcX \to C$ such that
\begin{enumerate}
\item Each closed fiber is an integral cubic surface.
\item The total space $\mcX$ has terminal singularities and is Gorenstein.
\end{enumerate}
\end{thm}

Such a model is called a standard model of $X/K(B)$ in \cite{CortiCubic}.

A. Corti gives an algorithm to produce such a model in \cite{CortiCubic} and shows that if the algorithm terminates, 
the end product is the so-called standard model satisfying the conditions in the above theorem (c.f. Theorem 2.15, \cite{CortiCubic}). 
When $K$ is a field of characteristic zero, the existence of such a model (i.e. the termination of the algorithm) is proved by Corti. 
The general case follows from Koll\'ar's result on the existence of semi-stable models \cite{KollarIntegralPolynomial}.

Essentially the only thing in \cite{WACubic} that needs to be changed is the proof of following lemma.
\begin{lem}[=Lemma 5.1 in \cite{WACubic}]\label{lem:smoothlocus}
Let $\pi: \mcX \to B$ be a standard model of families of cubic surfaces over a smooth projective curve $B$ 
and let $s: B \to \mcX$ be a section. 
Given finitely many points $b_1, \ldots, b_k$ in $B$, and a positive integer $N$, 
there is a section $s': B \to \mcX$ such that $s'$ is congruent to $s$ modulo $\mfm_{B,b_i}^{N}$ 
and $s'(B-\cup b_i)$ lies in the smooth locus of $\pi: \mcX \to B$.
\end{lem}
The author proves this lemma in characteristic $0$ in \cite{WACubic}. 
The proof given there uses the fact that a dominant map in characteristic $0$ is separable. 
Below we give a variant of the original proof which avoids the use of such a statement.
\begin{proof}
We can assume that the base field $K$ is uncountable.

We first resolve the singularities of $\mcX$ along the fibers over $b_i$ 
in such a way that the resolution is an isomorphism except along the singular locus in the fibers over $b_i$ (\cite{Resolution1}, \cite{Resolution2}). 
Then we use the iterated blow-up construction (c.f. Section 2.2 \cite{WACubic}) according to the jet data of $s$ near the points $b_i$. 
After sufficiently many iterated blow-ups, fixing the jet data is the same as passing through fixed components. 
Call the new space $\mcX_1$.

Then the lemma is reduced to showing that there is a section of the new family $\mcX_1 \to B$ 
which has desired intersection number with irreducible components of the fibers over $b_1, \ldots, b_k$ in $B$ 
and lies in the smooth locus of $\pi_1: \mcX_1 \to B$.

Denote by $f_1: B \to \mcX_1$ the strict transform of the given section $s: B \to \mcX$. The section $f_1$ has the correct intersection number but may intersect the singular locus of $\mcX_1$.
We will show below that if the section $f_1$ only passes through one singular point of $\mcX_1$, 
then we can deform the section away from the singularity. 

Assuming this, the general case can be proved by induction. 
Namely one first resolve all but one singularity along the section and then apply this argument to deform the section 
away from this singularity (after adding enough very free curves in general fibers). 
In this way one get a section which passes through less singular points. 
In this argument, we only take general deformations. 
So the condition of the intersection numbers is always preserved.

In the following we explain why a general deformation of $f_1$ deforms outside the singularities of $\mcX_1$.
We may also assume that the singularity on the section is the only singularity of the total space $\mcX_1$. 
Denote by $b$ the image of the singular point of $\mcX_1$ in $B$. 
Take a resolution of singularities $\pi_{21}: \mcX_2 \to \mcX_1$ which is an isomorphism over the smooth locus of $\mcX_1$
such that the exceptional locus in $\mcX_2$ consists of simple normal crossing divisors $E_i, i=1, \ldots, n$. 
After adding very free curves in general fibers and smoothing, we may assume the strict transform of the section $f_1$, 
denoted by $f_2: B \to \mcX_2$, 
satisfies $H^1(B, \mcN_{f_2}(-p))=0$ for any point $p \in B$, 
where $\mcN_{f_2}$ is the normal sheaf of the section in $\mcX_2$. 
This in particular implies that $\mcN_{f_2}$ is globally generated and the deformation of the section is a unobstructed. 

Let $V$ be an irreducible component of the Kontsevich moduli space of stable maps of $\mcX_1$ containing the point represented by the map $f_1: B \to \mcX_1$. 
There is a forgetful map $F$ from the Kontsevich space of stable sections of the fibration $\mcX_2 \to B$ 
to the Kontsevich space of stable sections of the fibration $\mcX_1 \to B$. 
Every section of $\mcX_1 \to B$ lifts to a section of $\mcX_2 \to B$, 
so the forget map $F$ is surjective on closed points. 
Since the field $K$ is uncountable and the Kontsevich space has only countably many irreducible components, 
there is an irreducible component $U$ of the moduli space of stable sections which maps surjectively onto $V$. 
Furthermore, a general point of $U$ parametrizes a section of $\mcX_2 \to B$ (i.e. the domain is irreducible). 
Let $f_2': B \to \mcX_2 $ be a section parameterized by a very general point in $U$ 
and $f_1'$ be the composition $\pi_{21} \circ f_2'$ of $f_2': B \to \mcX_2$ and $\pi_{21}: \mcX_2 \to \mcX_1$. 

If $f_1': B \to \mcX_1$ already avoids the singular locus, we are done. 
In the following, assume that $f_1'(B)$ still passes through the (unique) singular point of $\mcX_1$.

Since the map $F: U \to V$ is surjective, 
there is a stable map $\tilde{f}_2$ from a possibly reducible domain to $\mcX_2$ 
whose composition with $\pi_{21}$ is the section $f_1: B \to \mcX_1$. 
We claim that the domain of $\tilde{f}_2$ has to be reducible. 
Assume the contrary, then the stable map $f_2'$ is a deformation of $\tilde{f}_2=f_2$. 
Thus a general point of $U$ is also unobstructed, in particular a smooth point of $U$, 
and $U$ has the expected dimension $-{f_2'}^*K_{\mcX_2} \cdot B$ at this point.

The standard model $\mcX$ has $3$-fold terminal and local complete intersection singularities. 
So does the new total space $\mcX_1$ by construction. 
Therefore every irreducible component containing the point $f_1': B \to \mcX_2 \to \mcX_1$ 
has dimension at least $-f_1^*K_{\mcX_1}\cdot B$. 
Furthermore, by definition of terminal singularities, we have
\[
-K_{\mcX_1}=-K_{\mcX_2}+\sum_{i=1}^n a_i E_i, a_i >0,
\]
where the sum is over all exceptional divisors of $\pi_{21}: \mcX_2 \to \mcX_1$.

Since the image of $f_1'(B)$ in $\mcX_1$ intersects the singular locus, 
$-{f_1'}^*K_{\mcX_1} \cdot B$ is strictly larger than $-{f_2'}^*K_{\mcX_2} \cdot B$. 
Hence $\dim V > \dim U$, which is impossible since $U$ surjects onto $V$.

Denote by $\tilde{f}_2: \tilde{B}=B \cup R \to \mcX_2$ 
the stable map from a reducible domain whose composition with $\pi_{21}$ is the section $f_1$. 
Note that the resolution of singularities $\mcX_2 \to \mcX_1$ is an isomorphism away from the fiber over the point $b$. 
So the curve $R$ is supported in the exceptional divisors of the fiber of $\mcX_2$ over the point $b$.

Let $H_1$ be an ample divisor on $\mcX_1$ and $H_2=\pi_{21}^*H_1$. 
Then there are positive rational number $b_1, \ldots, b_n$ such that $H_2-\sum b_i E_i$ is an ample divisor on $\mcX_2$. 
Up to perturbing the numbers $b_i$ and renumbering the index, 
we may assume that $b_1 < b_2< \ldots <b_n$. The section $f_2(B)$ intersects $E_k$ for some $k$. 
Assume that $f_2'(B)$ intersects the divisor $E_{k'}$. 
We have the following inequality 
\begin{align*}
&{f_1'}^*H_1 \cdot B-b_{k'}={f_1'}^*H_1 \cdot B-{f_2'}^*(\sum b_i E_i) \cdot B\\
=&{f_2'}^*H_1 \cdot B-{f_2'}^*(\sum b_i E_i) \cdot B ~(\text{since }{f_1'}^*H_1 \cdot B={f_2'}^*H_2 \cdot B)\\
=&{\tilde{f}_2}^*(H_2-\sum b_i E_i) \cdot (B+R)\\
&(\text{since } \tilde{f}_2:B\cup R \to \mcX_1 \text{ and } f_2':B \to \mcX_1 \text{ are deformation equivalent})\\
=&f_2^*(H_2-\sum b_i E_i) \cdot B +(H_2-\sum b_i E_i) \cdot R\\
>&f_2^*(H_2-\sum b_i E_i) \cdot B=f_2^*H_2 \cdot B-b_{k}.
\end{align*}
Since ${f_1'}^*H_1 \cdot B=f_1^*H_1 \cdot B=f_2^*H_2 \cdot B$, we have $b_{k'}<b_k$. Thus $k'<k$.

To sum up, we start with a section $f_1: B \to \mcX_1$ whose strict transform $f_2: B \to \mcX_2$ is an unobstructed section  
which intersects the exceptional divisor $E_k$. 
If the deformation of the section $f_1:B \to \mcX_1$ in the irreducible component $V$ still intersect the singular locus, 
then we produce a new section $f_2': B \to \mcX_2$ which intersect the exceptional divisor $E_{k'}$ for some $k'<k$. 
Clearly in the process we may keep all the desired intersection numbers unchanged.

Continue this process, we will eventually find a section $s': B\to \mcX_1$ 
which has the desired intersection numbers and lies in the smooth locus of the total space $\mcX_1$. 
Finally note that if a section lies in the smooth locus of the total space $\mcX_1$, 
then the section lies in the smooth locus of the morphism $\pi: \mcX_1 \to B$.
\end{proof}

Once this lemma is proved, the proof proceeds exactly as in \cite{WACubic}. 
We first produce a multisection of degree $2$ and reduce the weak approximation problem to a new weak approximation problem for the multisection. 
In particular by Lemma \ref{lem:smoothlocus} we only need to approximate formal sections in the smooth locus, which is handled in Lemma 4.5 \cite{WACubic}. 
The proof of this proposition depends on three things: 
strong rational connectedness of the smooth locus of a cubic surface with ADE singularities (Lemma \ref{lem:surface}), 
computation of base change and birational modifications of the integral model (Proposition 3.4 \cite{WACubic}), 
and G-equivariant techniques (Lemma 3.6, Theorem 4.1 \cite{WACubic}). 
When the characteristic is not $2, 3, 5$, the proof of these results needs no change. 
In Proposition 3.4 \cite{WACubic}, the author computed the base change needed for the new central fiber to have ADE singularities. 
They are of degrees $2, 3, 4, 5, 6$. Thus under the assumption on the characteristic, 
all the base changes needed are Galois and the Galois groups are cyclic groups of order prime to the characteristic. 
The $G$-equivariant techniques apply in these cases. 

\section{Fundamental group of rationally connected fibrations}\label{sec:fundamentalgroup}

In this section we collect some easy corollaries of Koll\'ar's results on the fundamental group of (separably) rationally connected varieties. All the fundamental groups mentioned below are to be understood as the algebraic fundamental group.
\begin{thm}\label{thm:fungroup}
Let $\pi: \mcX \to B$ be projective family of varieties over a smooth projective connected curve $B$ defined over a field $k$ and $x \in \mcX$ a $k$-rational point in the smooth locus of $\mcX$. 
Assume that the generic fiber is smooth and separably rationally connected. 
Furthermore assume that there is a free section $s_0: B \to \mcX$ (c.f. Definition \ref{def:free}).
Then there is a geometrically irreducible component of the space of sections with a marked point defined over $k$, i.e. a family $\mathcal{S} \to W$ together with a section $p: W \to \mathcal{S}$ and an evaluation morphism $ev: \mathcal{S} \to \mcX$ such that
\begin{enumerate}
\item
$ev(p(W))=x$.
\item
A general geometric point $w$ in $W$ parameterize a $2$-free section $\mathcal{S}_w$.
\item
Choose an algebraic closure $K$ of $k$. For any open subvariety $\mathcal{X}^0 \subset \mcX$ defined over $K$ containing the point $x$, the map of fundamental groups
\[
\pi_1(ev^{-1}(\mathcal{X}^0), p(w)) \to \pi_1(\mathcal{X}^0, x)
\]
is surjective.
\end{enumerate}
\end{thm}
\begin{proof}
This is a simple corollary of Koll\'ar's result on the fundamental group of separably rationally connected varieties.

\begin{thm}[\cite{KollarFundamentalGroup}, (3)-(5)]
Let $X$ be a smooth projective separably rationally connected variety defined over a field $k$ and $x \in X$ a $k$-rational point. Then there is a dominating family of rational curves through $x$ (defined over $k$):
\[
F: U \times \PP^1 \to X, F(U \times [0, 1])=x
\]
with the following properties:
\begin{enumerate}
\item 
The variety $U$ is geometrically irreducible, smooth and open in $\text{Hom}(\PP^1, X, [1, 0] \mapsto x)$. The morphism $F: U \times [1, 0] \to X$.
\item
For every geometric point $u$ of $U$, $F_u^*T_X$ is ample.
\item
Choose an algebraic closure $K$ of $k$. For any $K$-open subvariety $X^0 \subset X$ containing the point $x$, the map 
\[
\pi_1(F^{-1}(X^0) \cap (U \times \PP^1), (u, [0, 1])) \to \pi_1(X^0, x)
\]
is surjective.
\item
Under the same assumptions as before, for any \'etale cover $Y^0 \to X^0$ defined over $K$, there is an open subset $U(Y_0)$ of $U$ such that for any $K$-point $u \in U(Y_0)$, the fiber product $\PP^1_u \times_X Y^0$ is geometrically irreducible.
\end{enumerate}
\end{thm}
The last statement is not explicitly stated in \cite{KollarFundamentalGroup} but is very easy to deduce, see, for example the proof of (\ref{ind1}). 

We apply this theorem to the generic fiber of the fibration with the rational point given by the section. 
Then we get a family of very free curves defined over the function field of $B$, which can be ``spread out" to a geometrically irreducible family, still denoted by $U$, of rational curves in general fibers passing through the section. 
Over a general fiber, the family of rational curves in $U$ still satisfies the conditions in Koll\'ar's theorem, if we choose the base point to be the intersection point of the section $s_0$ with the fiber.

Consider the unique irreducible component $\mathcal{S} \to W$ containing the union of the free section $s_0$ and sufficiently many rational curves of this family in different general fibers. 
A general deformation of the union is a $2$-free section.
After shrinking the base, we may assume that the total space $\mathcal{S}$ is normal and $W$ still contains all the points which parameterize a stable map consisting of the union of $s_0$ and general very free curves of the family $U$ in general fibers.

Given an open subset $\mcX^0 \subset \mcX$, if the induced map on fundamental groups 
\[
\pi_1(ev^{-1}(\mathcal{X}^0), p(w)) \to \pi_1(\mathcal{X}^0, x)
\]
is not surjective, then there is a finite \'etale cover $\mathcal{Y}^0 \to \mathcal{X}^0 \subset \mcX$ from an irreducible variety $\mathcal{Y}^0$ defined over $K$, 
such that there is a point $y \in \mathcal{Y}^0$ which is mapped to $x$ under the morphism, 
and the fiber product $\mathcal{S} \times_{\mcX} \mathcal{Y}^0$ is reducible.
The total family of $\mathcal{S}$ is normal, so is the fiber product.
Thus it is disconnected.

We choose a reducible fiber $\mathcal{S}_0$ in the family consisting of the free section $s_0$ and general very free curves in general fibers in the following way.
We first choose general points $b_1, \ldots, b_k \in B$ so that the points $s_0(b_i)$ is contained in the open subset $\mcX^0$ and the family of very free curves in the generic fiber specializes to a family of very free curves in the fibers over $b_1, \ldots, b_k$ which induces surjections on the fundamental group.
The cover $\mathcal{Y}^0 \to \mathcal{X}^0$ induces a finite (possibly disconnected) \'etale cover $\mathcal{Y}^0_{b_i} \to \mcX^0_{b_i}$ for each $1 \leq i \leq k$.
We then choose a very free curve in the fiber over $b_i$ to be a general curve $C_i$ such that the base change $C_i \times \mathcal{Y}^0_{b_i}$ has the same number of geometrically irreducible components as geometrically connected components of $\mathcal{Y}^0_b$.
That is to say, the base change of each geometrically irreducible component of $\mathcal{Y}^0_{b_i}$ to $C_i$ is geometrically irreducible.
This is possible by Koll\'ar's results.

We now look at the fiber product $\mathcal{S}_0 \times_{\mcX} \mathcal{Y}^0$. This is geometrically connected. To see this, simply note that each geometrically irreducible component of $C_i \times \mathcal{Y}^0_{b_i}$ is connected by one (in fact, any) geometrically irreducible component of the inverse image of the section $s_0$ and that all the geometrically irreducible components of the inverse image of the section $s_0$ are connected by the base change of any curve $C_i$.
So it connects every geometrically irreducible component of $\mathcal{S} \times_{\mcX} \mathcal{Y}^0$.
Thus we get a contradiction and the map on fundamental groups
\[
\pi_1(ev^{-1}(\mathcal{X}^0), p(w)) \to \pi_1(\mathcal{X}^0, x)
\]
is surjective.
\end{proof}

We also have the following version without specifying a base point.

\begin{thm}\label{thm:b2}
Let $\pi: \mcX \to B$ be projective family of varieties over a smooth projective connected curve $B$ defined over a field $k$.
Assume that the generic fiber is smooth and separably rationally connected. 
Furthermore assume that there is a free section $s_0: B \to \mcX$.
Then there is a geometrically irreducible component of the space of sections defined over $k$, 
i.e. a family of sections $\mathcal{S} \to W$ and an evaluation morphism $ev: \mathcal{S} \to \mcX$ such that
\begin{enumerate}
\item
A general geometric point $w$ in $W$ parameterize a $2$-free section $\mathcal{S}_w$.
\item
Choose an algebraic closure $K$ of $k$. For any dominant map $f: \mathcal{Z} \to \mcX$ from an irreducible variety $\mathcal{Z}$, there is an open subset $W^0$ of $W$ such that for any geometric point $w \in W^0$, the fiber product $\mathcal{S}_w \times_{\mcX} \mathcal{Z}$ is geometrically irreducible.
\end{enumerate}
\end{thm}

\begin{proof}
The free section $s_0$ determines a geometrically irreducible family of sections $\mathcal{S} \to W$ defined over $k$. 
Moreover there is a dominant map $\mathcal{S} \to \mcX$.
Let $K$ be the function field of $\mathcal{S}$ and $\eta$ be the generic point of $\mathcal{S}$. 
Apply Theorem \ref{thm:fungroup} over the field $K$ to the $K$-point $\eta$ of $\mcX \times_k K$.
For details see the proof of (6) in \cite{KollarFundamentalGroup}.
\end{proof}

In the following we specialize to the case of a family Fano complete intersections satisfying Hypothesis \ref{hyp}. For the ease of the reader, we reproduce the hypothesis below.

\begin{hypo}[=Hypothesis \ref{hyp}]
Given a family $\mcX \to B$ of Fano complete intersections defined over an algebraically closed field, assume the followings are satisfied.
\begin{enumerate}
\item There is free section $s: B \to \mcX$.
\item The Fano scheme of lines of a general fiber $\mcX_b$ is smooth.
\item A general line in a general fiber is a free line.
\item The relative dimension of $\mcX \to B$ is at least $3$.
\end{enumerate}
\end{hypo}
To apply the results of Koll\'ar, we first need to check that a general fiber is separably rationally connected. This is taken care of by the following result.
\begin{thm}[\cite{hypersurface}, Corollary 9]
Let $X$ be a smooth Fano complete intersection of dimension at least $3$. Then $X$ is separably rationally connected if and only if it is separably uniruled.
\end{thm}

Thus all the previous results apply to the fibration $\mcX \to B$. 
Recall that $F \to B$ is the relative Hilbert scheme of lines and $F(U) \to U$ is the relative Hilbert scheme of lines for the smooth fibers. 
Finally $\overline{F} \to B$ is the closure of $F(U)$ in $F$ and $\mathcal{L} \to \overline{F}$ is the universal families of lines restricted to $F$.
There is a natural $B$-morphism $\mathcal{L} \to \mcX$.

\begin{lem}\label{lem:familyofnicesection}
Let $\mcX \to B$ be a family defined over an algebraically closed field $k$ satisfying Hypothesis \ref{hyp}. 
\begin{enumerate}
\item\label{ind1}
There is a nice section.
\item \label{general}
Let $\mathcal{S} \to W$ be an irreducible component of the space of sections such that there is a geometric point $w \in W$ which parameterizes a nice section $\mathcal{S}_w$. Then a general point of $W$ parameterizes a nice section.
\item \label{induction}
Let $\mathcal{S} \to W$ be a geometrically irreducible component of the space of sections such that a general geometric point parameterizes a nice section. Then $\mathcal{S} \times_{\mcX} \mathcal{L}$ is geometrically irreducible and generically smooth.
Furthermore it is contained in a unique irreducible component of the Kontsevich moduli space of stable sections which contains an open substack parameterizing nice sections.
\end{enumerate}
\end{lem}

\begin{proof}
The morphism $\mathcal{L} \to \mcX$ factors through a variety $\mathcal{Z}$: $\mathcal{L} \to \mathcal{Z} \to \mcX$ such that a general fiber of $\mathcal{L} \to \mathcal{Z}$ is geometrically irreducible and $\mathcal{Z} \to \mathcal{X}$ is finite and generically \'etale.
Let $\mcX^0$ be the open locus of $\mcX$ such that $\mathcal{Z} \to \mathcal{X}$ is \'etale and $\mcX^1$ be the open locus of $\mcX$ such that $\mathcal{Z} \to \mathcal{X}$ is has constant fiber dimension.
Let $\mathcal{Z}^0$ be the inverse image of $\mcX^0$ in $\mathcal{Z}$.
The complement of $\mcX^1$ in $\mcX$ has codimension at least $2$.
Thus a general deformation of a free section lies in $\mcX^1$.

Choose a general point $x$ in $\mcX^0$. Consider the family of sections $\mathcal{S} \to W$ containing the point $x$ constructed in Theorem \ref{thm:fungroup}.
By the Theorem \ref{thm:fungroup}, the fiber product $\mathcal{S} \times_{\mcX} \mathcal{Z}^0$ is geometrically irreducible. 
By shrinking $W$, we may assume that $W$ is smooth.
Also the morphism $\mathcal{S} \to W$ is a smooth morphism  
Furthermore there is a section from $W$ to $\mathcal{S} \times_{\mcX} \mathcal{Z}^0$ by choosing a point in $\mathcal{Z}^0$ lying over $x$.
The generic fiber of $\mathcal{S} \times_{\mcX} \mathcal{Z}^0 \to W$ is smooth and contains a rational point, thus geometrically irreducible.
Then a general fiber of the morphism $\mathcal{S}\times_{\mcX} \mathcal{Z}^0 \to W$ is geometrically irreducible.

A $2$-free section $s: B \to \mcX$ is nice if 
\begin{enumerate}
\item
It is contained in $\mathcal{X}^1$.
\item
The fiber product $B \times_{\mcX} \mathcal{Z}^0$ is irreducible.
\end{enumerate}

So a general member of the family $\mathcal{S} \to W$ is a nice section.

For the second statement, let $x$ be a general point in $\mathcal{S}_w \cap \mcX^0$. We first deform the nice section $\mathcal{S}_w$ with one general point fixed. 
This gives a family $\mathcal{T}\to U, F: \mathcal{T} \to \mcX$, with a section $p: U \to \mathcal{T}$ such that $F(p(U))=x$.
The deformation covers an open subset of $\mcX$ since $\mathcal{S}_w$ is $2$-free.
Furthermore a general deformation is a $2$-free section which lies in the locus $\mcX^1$.

Consider the fiber product $\mathcal{T} \times_{\mcX} \mathcal{Z}^0$ and its projection to $U$. 
A section $T_u$ over a general geometric point $u \in U$ is $2$-free. 
Thus by replacing $U$ with a smaller open subset, we may assume that $\mathcal{T}, U$ are smooth and the fibration $\mathcal{T} \to U$ is smooth.
So $\mathcal{T} \times_{\mcX} \mathcal{Z}^0$ is irreducible if and only if it is connected.
There is a section of the map $\mathcal{T} \times_{\mcX} \mathcal{Z}^0 \to U$ by lifting $x$ to one of its inverse images in $\mathcal{Z}^0$.
This determines a geometrically irreducible component of $\mathcal{T} \times_{\mcX} \mathcal{Z}^0$. 
Since $\mathcal{S}_w \times_{\mcX} \mathcal{L}$ is irreducible, $\mathcal{S}_w \times_{\mcX} \mathcal{Z}^0$ is irreducible. 
Thus it lies in a unique irreducible component.
Then there can be only one irreducible component, independent of the chosen inverse image of $x$.
Once we know the total space is geometrically irreducible, a similar argument as above shows that for a general point $u \in U$, the section $\mathcal{T}_u$ is nice.

Consider the evaluation map of the total family $\mathcal{S} \to W$ to $\mcX$, $ev_S: \mathcal{S} \to \mcX$. 
The previous paragraphs shows that for a general point $x \in \mathcal{S}_w \cap \mathcal{X}^0$, there is one irreducible component of the fiber of $ev_S$ over the point $x$, whose general points parameterize nice section. 
The deformation of these cover an open subset of $\mcX$.
Then repeating the above argument shows that for a general point $y \in \mcX$, there is an irreducible component of the fiber of $ev_S$ over $y$ whose general points parameterize nice sections. 

The locus in $W$ such that the fiber of $\mathcal{S} \times_{\mcX} \mathcal{Z}^0 \to W$ is geometrically irreducible is a constructible set by EGA IV 9.7.9 \cite{EGAIV}.
Then the previous paragraph shows that it is contains an open subset of $W$.

For the last statement, first note that a standard computation using the exact sequence of normal sheaves shows that a general deformation of the union of $\mathcal{S}_w$ and a free line is a $2$-free section.
Furthermore, if we deform the union of a nice section and a free line with one general point of the section fixed, then a general deformation is a nice section.
Call this family $\mathcal{C} \to V$.
We may assume $V, \mathcal{C}$ are smooth after restricting to a smaller open subset.
The proof proceeds similarly to proof of Theorem \ref{thm:fungroup}. 
Namely one first shows that the total space of the base change $\mathcal{C} \times_{\mcX} \mathcal{Z}^0$ is geometrically irreducible.
One can show this by specializing to the union of a nice section and a general free line.
Then a general member of the family is a nice section.
So a general point of the family parameterizes a nice section by the second statement.
\end{proof}

\bibliographystyle{alpha}
\bibliography{MyBib}

\end{document}